\documentclass[11pt,reqno]{amsart}

\usepackage{amsmath,amssymb,amsthm, amsfonts}   % math
\usepackage{bbm} % So that \mathbb{1} works if written \mathbbm{1}.
\usepackage{graphicx}   % need for pictures
\usepackage{pstricks}   % for color and drawings
\usepackage[colorlinks=true, linkcolor=blue, citecolor=magenta]{hyperref} % use for hypertext links
\usepackage{tasks}
\usepackage{enumerate}
\usepackage{caption,subcaption} % To place to figures side by side inside a single figure.

\usepackage{mathtools}      % This stuff is so that the equations
\mathtoolsset{showonlyrefs} % only show numbers if referenced.

\usepackage{mathrsfs}
\usepackage{multicol}   % for 2-column or 3-column. Use \begin{multicols}{2} or whatever,
\usepackage{fullpage}   % makes the margines way smaller

\usepackage[foot]{amsaddr}

\newtheorem{thm}{Theorem}[section]
\newtheorem{cor}[thm]{Corollary}
\newtheorem{prop}[thm]{Proposition}
\newtheorem{lem}[thm]{Lemma}

\newtheorem{alg}[thm]{Algorithm}
\newtheorem*{acknowledgements*}{Acknowledgements}
\newtheorem*{thm*}{Theorem}

\theoremstyle{definition}

\theoremstyle{remark}
\newtheorem{rem}[thm]{Remark}

\title{%The geometry of Riemann's non-differentiable function - Part I: Asymptotic behaviour and Hausdorff dimension
On the Hausdorff dimension of Riemann's non-differentiable function
}
\author{Daniel Eceizabarrena}
\address{University of Massachusetts Amherst - Department of Mathematics and Statistics. \newline 
Lederle Graduate Research Tower 1623D, 710 N Pleasant St, 
Amherst, MA 01003-9305, USA \newline 
 e-mail: {\tt eceizabarrena@math.umass.edu} \newline
 }
\begin{document}

\begin{abstract}
Recent findings show that the classical Riemann's non-differentiable function has a physical and geometric nature as the irregular trajectory of a polygonal vortex filament driven by the binormal flow. In this article, we give an upper estimate of its Hausdorff dimension. We also adapt this result to the multifractal setting. 
To prove these results, we recalculate the asymptotic behavior of Riemann's function around rationals from a novel perspective, underlining its connections with the Talbot effect and Gauss sums, with the hope that it is useful to give a lower bound of its dimension and to answer further geometric questions.
\end{abstract}

\maketitle

%\tableofcontents

%This article is part of the series of papers by the author: it is preceded by  \cite{Eceizabarrena2019} and followed by \cite{Eceizabarrena2020}. 

\section{Introduction}

\subsection{Riemann's non-differentiable function}

% Considerably changed on 4th September 2019

In a lecture in the Royal Prussian Academy of Sciences in 1872, in Berlin, Weierstrass \cite{Weierstrass1872} explained against the belief of the time that a continuous function need not have a well-defined derivative, proposing the famous Weierstrass functions,
\begin{equation}\label{eq:Weierstrass_Function}
W(x) = \sum_{n=1}^\infty a^n\,\cos (2\pi b^k x), \qquad 0 < a < 1, \quad b>1, \quad ab \geq 1,
\end{equation}
as counterexamples. However, his main motivation to tackle this problem was the function
\begin{equation}\label{RiemannFunctionOriginal}
R(x) = \sum_{n=1}^{\infty}{\frac{ \sin{ (n^2x ) } }{ n^2 }}
\end{equation}
proposed by Riemann some years earlier. Riemann is believed to have claimed that $R$ was continuous but nowhere differentiable. Even if no written nor oral proof survived, \eqref{RiemannFunctionOriginal} became widely known as Riemann's non-differentiable function. Weierstrass claimed that this conjecture was a \textit{somewhat difficult} problem, and he was correct indeed, since one century had to pass until Gerver \cite{Gerver1970} disproved the conjecture in 1970. He showed that $R$ is differentiable at points $\pi x$ where $x\in\mathbb{Q}$ is a quotient of two odd numbers, with derivative equal to $-1/2$. Previously, in 1916, Hardy \cite{Hardy1916} had shown that $R$ is not differentiable in $\pi x$ if $x$ is irrational. The problem was completely solved in 1971 by Gerver himself \cite{Gerver1971}, showing that it was also the case of the remaining rationals. Later, Duistermaat \cite{Duistermaat1991}, Jaffard \cite{Jaffard1996} and Jaffard and Meyer \cite{JaffardMeyer1996} studied the regularity of $R$ deeper. 
%Results with a geometric flavour were also obtained by Chamizo and C\'ordoba \cite{ChamizoCordoba1999}, who studied the Minkowski dimension of the graph of $R$ and some other Fourier series.
In all these works, a common technique is to study a generalization of $R$ to the complex plane,
\begin{equation}\label{PhiDuistermaat}
\phi_D(t) = \sum_{n=1}^{\infty}{ \frac{e^{i \pi n^2 t}}{i \pi n^2} }, %\qquad \qquad \operatorname{Re}\phi_D(t) = \frac{1}{\pi}R(\pi t).
\end{equation}
for which $\operatorname{Re}\phi_D(t) = R(\pi t)/\pi$. 

\subsection{A physical and geometric version of Riemann's function}

Recently, De la Hoz and Vega \cite{delaHozVega2014} found a version of Riemann's non-differentiable function, 
\begin{equation}\label{Phi}
\phi(t) = \sum_{k\in\mathbb{Z}}{\frac{ e^{ - 4\pi^2 i k^2 t } - 1 }{ -4\pi^2k^2 } },
\end{equation}
 in a novel context concerning the evolution of vortex filaments, thus giving it a fantastic geometric and physical interpretation. They showed that \eqref{Phi}, which is related to the previous $\phi_D$ by 
\begin{equation}\label{FromPhiDuistermaatToPhi}
\phi(t) = -\frac{i}{2\pi}\phi_D(-4\pi t) + i t + \frac{1}{12}, \qquad \qquad \forall t \in \mathbb{R},
\end{equation}
approximates accurately the trajectories of the corners of polygonal vortex filaments that follow the binormal flow, a model for the evolution of a single vortex filament that is represented by the vortex filament equation (VFE) or localized induction approximation (LIA), 
\begin{equation}\label{VFE}
\boldsymbol{X}_t = \boldsymbol{X}_s \times \boldsymbol{X}_{ss}, \qquad \text{ or equivalently } \qquad \boldsymbol{X}_t =  \kappa \, \boldsymbol{B}.
\end{equation}
Here, the vortex is represented by the curve $\boldsymbol{X} : \mathbb{R}^2 \to \mathbb{R}^3$ with variables $s$ and $t$, the arclength and the time respectively, and is given an initial condition $\boldsymbol{X}(s,0)$. Also, $\kappa = \kappa(s,t)$ represents the curvature and $\boldsymbol{B} = \boldsymbol{B}(s,t)$ is the binormal vector.

The VFE was originally proposed by Da Rios \cite{DaRios1906}, though forgotten and rediscovered many times by different authors, as discussed in \cite{Ricca1991}. A landmark result in the study of this equation is due to Hasimoto \cite{Hasimoto1972}, who established a direct connection between the VFE and the cubic nonlinear Schr\"odinger equation (NLS). The relationship works as follows: let $\kappa$ and $\tau$ be the curvature and torsion of the filament $\boldsymbol{X}$ that evolves according to the VFE, and define the complex-valued function
\begin{equation}\label{Hasimoto}
\psi(s,t) = \kappa(s,t) \, e^{i\int_0^s{ \tau(\sigma,t)\,d\sigma } }.
\end{equation} 
This is often called the filament function. Hasimoto showed that $\psi$ satisfies 
\begin{equation}\label{NLS}
\psi_t = i \psi_{ss} + \frac{i}{2}\, \left( \left| \psi \right|^2 + A(t) \right)\psi,
\end{equation}
where $A(t)$ is a real function of time. This function $A(t)$ supposes no extra inconvenient in practice because the function $\Psi(s,t) = \psi(s,t)\, e^{-i/2 \int_0^t A(\tau)d\tau}$ solves the standard cubic NLS 
\begin{equation}
\Psi_t = i \Psi_{ss} + \frac{i}{2}\left| \Psi \right|^2\Psi.
\end{equation}
The usefulness of this transformation is evident because, under the condition that it can be unmade, it allows to work directly with the cubic NLS. In principle, if $\psi$ is found, its definition yields $\kappa$ and $\tau$ directly and the tangent vector is obtained integrating the Frenet-Serret system. The curve is then recovered integrating the tangent. Unfortunately, it is not always trivial to materialize these ideas. Even in the simple case of a partially straight filament with $\kappa = 0$, the Frenet-Serret frame is not well-defined! In fact, Hasimoto needed to assume this non-vanishing restriction for the curvature. However, Koiso \cite{Koiso1997} showed that a parallel frame can used instead of the classic Frenet-Serret frame to remove this restriction, unmake the transformation and recover $\boldsymbol{X}$.

We are particularly interested in the evolution of closed vortex filaments. Think of smoke rings of cigarettes which, as we know, essentially maintain their shape while they travel. But what happens if the ring has the shape of a triangle? 
%This question is connected to the study of jets with polygonal nozzles \cite{GutmarkGrinstein1999}.
 In  \cite{KlecknerScheelerIrvine2014} they did this experiment with a clover-shaped filament, and its evolution is nothing close to that of the circular ring. De la Hoz and Vega \cite{delaHozVega2014} then showed that the triangle behaves in a similar way. More generally, they studied general regular polygonal vortices, and they showed that surprisingly their evolution is ruled by the Talbot effect, an originally optical phenomenon. A numeric simulation of the evolution of the triangular vortex is available in \cite{KumarVideos} or in the video \url{https://youtu.be/f3HQFfTtFtU} by Sandeep Kumar. 
 
 The video above also shows the trajectory of one of the corners of the triangle. These trajectories were also numerically simulated in \cite[Figure 2]{delaHozVega2014}, which turn out to be plane and some of which are shown in Figure~\ref{fig:Numeric_Trajectories}. Comparing them to the image of $\phi$  \eqref{Phi} shown in Figure~\ref{FIG_Curva}, there is little doubt that this version of Riemann's non-differentiable is a very good approximation of these trajectories.
 \begin{figure}[h]
 \includegraphics[width=\textwidth]{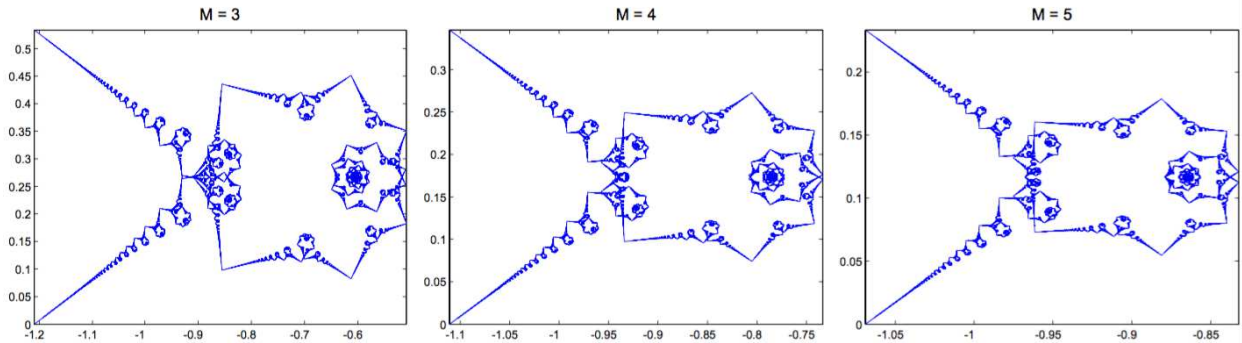}
 \caption{Numeric simulations of the trajectory of a corner of the $M$-sided regular polygon, for $M=3,4,5$. Image by F. De la Hoz and L. Vega.}
 \label{fig:Numeric_Trajectories}
 \end{figure}
 \begin{figure}[h]
\centering
\includegraphics[width=0.5\linewidth]{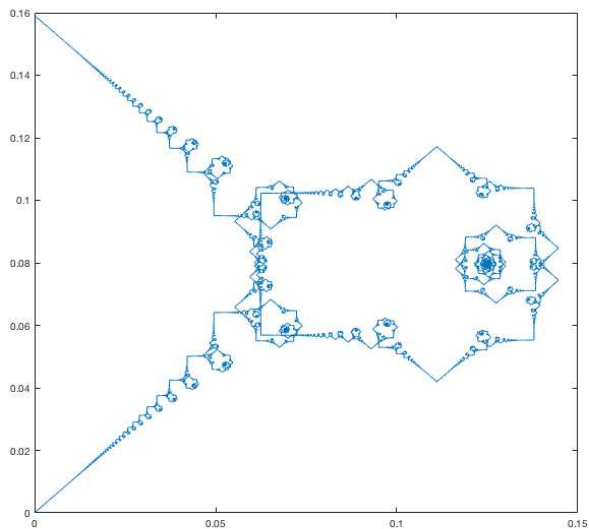}
\caption{\small The set $\phi([0,1/(2\pi)]) \subset \mathbb C$. The resemblance to the numeric trajectories in Figure~\ref{fig:Numeric_Trajectories} is astonishing.}
\label{FIG_Curva}
\end{figure}
 
 Let us briefly explain why Riemann's function appears in this context. For that, we need to describe the evolution of polygonal vortices with the VFE. Let $M \in \mathbb N$ and $\boldsymbol{X}_M$ be the solution to the VFE when the initial datum $\boldsymbol{X}_M(s,0)$ is a planar regular polygon of $M$ sides. An option to parametrize it is to do it first in the interval $[0,2\pi)$ and then to extend it periodically to $\mathbb R$, so that the problem becomes periodic in space. Thanks to Hasimoto's transformation, we can work with the filament function $\psi_M$ \eqref{Hasimoto} instead, so we need to parametrize the curvature and the torsion of the polygon. The torsion is zero because the polygon is planar. Regarding the curvature, we may think of each corner as a Dirac delta, so placing $M$ of them uniformly in $[0,2\pi)$ and extending periodically, it is reasonable to set
\begin{equation}\label{RegularPolygonM}
\psi_M(s,0) = \kappa_M(s,0) = \frac{2\pi}{M}\sum_{k \in \mathbb{Z}}{ \delta\left( s - \frac{2\pi}{M}k \right) }. 
\end{equation} 
We now do heuristic but clarifying computations. Instead of solving NLS for $\psi_M$, forget about the nonlinearity and assume  $\psi_M$ solves the free Schr\"odinger equation 
\begin{equation}\label{FreeSchrodingerEquation}
\psi_t = i\,\psi_{ss}
\end{equation} 
%\begin{equation}\label{FreeSchrodingerEquation}
%\left\{
%\begin{array}{ll}
%\psi_t = i\,\psi_{ss} &  \text{ in } \mathbb{R}^2, \\
%\psi(s,0) = \psi_0(s) =  \sum_{k \in \mathbb{Z}}{ \delta\left( s - k \right) }, &   s \in \mathbb{R} .
%\end{array}
%\right.
%%\end{equation} 
With the help of the Poisson summation formula, the well-known solution is
\begin{equation}\label{FreeSchrodingerSolution}
\psi_M(s,t) = e^{it\partial_s^2}\left(\frac{2\pi}{M}\,\sum_{k \in \mathbb{Z}}\delta(\cdot - \textstyle{\frac{2\pi}{M}}\, k)  \right)(s) =  \sum_{k\in\mathbb{Z}}{ e^{ i M k s -  i M^2 k^2 t} }.
\end{equation}
To recover $\boldsymbol{X}_M$, we should integrate the Frenet-Serret system in $s$ to get the tangent, and integrate the latter also in $s$. Again, a very heuristic shortcut is to integrate $\psi_M$ twice in $s$, and since $\psi_M$ solves the free Schr\"odinger equation, that amounts to integrate it once in $t$. Thus, we would get
\begin{equation}\label{PhiIntegral}
\boldsymbol{X}_M(s,t) \approx i \, \int_0^t{ \sum_{k\in\mathbb{Z}}{ e^{ i M k s -  i M^2 k^2 \tau} } \, d\tau }.
\end{equation}  
The point $\boldsymbol{X}_M(0,0)$ represents a corner, whose trajectory is $\boldsymbol{X}_M(0,t)$. According to the definition of $\phi$ in \eqref{Phi}, we get 
\begin{equation}
\boldsymbol{X}_M(0,t) \approx  \sum_{k\in\mathbb{Z}}{\frac{ e^{ - i M^2 k^2 t } - 1 }{ -M^2 \, k^2 } } = \frac{4\pi^2}{M^2}\,  \phi\left( \frac{M^2}{4\pi^2}\, t \right).
\end{equation}

In view of the resemblance of the numeric trajectories of Figure~\ref{fig:Numeric_Trajectories} and the image of $\phi$ in Figure~\ref{FIG_Curva}, this crude approximation is surprisingly precise. Moreover, the larger $M$, the better the matching, which suggests some kind of convergence of the trajectories $\boldsymbol{X}_M$ to $\phi$ when $M \to \infty$. The first result in this direction has been given recently by Banica and Vega \cite{BanicaVega2020-R} for initial polygons of $M = 2n + 1$ sides with a particular parametrization. For completeness, we reproduce their result here in a simplified way. To put ourselves in context, observe that the parametrization of the periodic data we considered in \eqref{RegularPolygonM} gives infinitely many loops around the polygon.
\begin{thm*}{(\cite[Theorem 1.1]{BanicaVega2020-R})}
Let $n \in \mathbb N$ and the planar regular polygon of $2n+1$ sides be parametrized by $\boldsymbol{X}_n(s,0)$, which gives a single loop to the polygon when $|s| \leq n$ with its corners located at the integers, and which escapes to infinity by two straight lines when $|s|>n$. Then, $\lim_{n \to \infty}n\boldsymbol{X}_n(0,t) = \phi(t)$
\end{thm*}

Hence, this theorem and the novel point of view gives Riemann's non-differentiable function an intrinsic geometric and physical nature that makes its study from these perspectives an interesting topic. For instance, related to physics and the theory of turbulence, it was shown in \cite{BoritchevEceizabarrenaVilaca2019} that it is intermittent. However, in this article we focus in geometric aspects.

\subsection{Geometric study of Riemann's function}

A quick look at Figure~\ref{FIG_Curva} is enough to be convinced of the geometric complexity of Riemann's non-differentiable function. Related to this, for instance, in \cite{Eceizabarrena2020} its geometric differentiability was analyzed. It is also quite natural to wonder whether this is a fractal or not; this is precisely the focus of this paper.

Questions about the dimension (either Hausdorff, Minkowski or others) of non-differentiable functions are popular. A famous, long-lived problem is to prove that the dimension of the graph of the Weierstrass function \eqref{eq:Weierstrass_Function} is $ 2 + \log(a)/\log(b)$, 
%$D_{a,b} = 2 + \log(a)/\log(b)$, 
as was conjectured by Mandelbrot \cite{Mandelbrot1977} in 1977. While the result for the Minkowski dimension was proved in 1984 \cite{KaplanMalletParetYorke1984}, the conjecture for the Hausdorff dimension resists, at least partially. Aside from a randomized version by Hunt \cite{Hunt1998}, the best known result known is by Shen in 2018 \cite{Shen2018}, who proved the conjecture for any $0<a<1$ and $b\in \mathbb{N}, b \geq 2$ using dynamical systems.  

Riemann's non-differentiable function is also an interesting case of study, and in the spirit of Weierstrass' words, even a more difficult one due to the slower convergence of the series. The main result in the literature is by Chamizo and Córdoba \cite{ChamizoCordoba1999,Cordoba2008}, who proved that the Minkowski dimension of the graph of the original function \eqref{RiemannFunctionOriginal} is 5/4. Concerning the Hausdorff dimension, to my knowledge, no result is known yet. 

The discoveries in the context of vortex filaments, though, make us focus on the image of the complex valued function \eqref{Phi} shown in Figure~\ref{FIG_Curva} rather than in the graph of the original function. The question about the dimension of $\phi$ is in principle more challenging than studying the dimension of the graph of the original Riemann's function. Indeed, in the case of a graph, we have a complete control of the speed of the curve in the direction of the abscissa, while the image of a parametric curve can move in the plane arbitrarily. Also, the fact that Figure~\ref{FIG_Curva} is not a graph makes it have plenty of self-intersections that make its study harder.

\subsection{Results}

In this paper, we give a first approach to computing the Hausdorff dimension of the image of $\phi$.
%the main result in this paper is an upper bound of the Hausdorff dimension of the image of $\phi$.
\begin{thm}\label{TheoremHausdorffDimension}
The Hausdorff dimension of the image of Riemann's non-differentiable function $\phi$ defined in \eqref{Phi} satisfies
\[  1 \leq \operatorname{dim}_{\mathcal{H}}{ \phi(\mathbb{R}) } \leq \frac43.   \]
\end{thm}

This theorem
%~\ref{TheoremHausdorffDimension} 
can be generalized to the context of multifractality, a very popular topic in the mathematical study of turbulence which deals with the local H\"older regularity of functions. Let us briefly introduce it. 
%Indeed, differentiability of $\phi$ having been determined, H\"older regularity is a natural continuation in the analysis of its regularity. 
For $\alpha \geq 0$, a function $f$ is said to be $\alpha$-H\"older in $x_0 \in \mathbb{R}$, and denoted $f\in C^{\alpha}(x_0)$, if there exists a polynomial $P$ with $\operatorname{deg}P \leq \alpha$ such that 
\[ \left| f(x_0+h) - P(h)  \right| \leq C |h|^{\alpha}, \qquad \text{ when } h \text{ is small}.  \]
%When $\alpha <1$, the condition is nothing but $\left| f(x_0+h) - f(x_0)  \right| \leq C |h|^{\alpha}$. 
The H\"older exponent of $f$ at a given point $x_0$ is the maximal H\"older regularity of $f$ at $x_0$,
\begin{equation}\label{MaximalHolderRegularity}
\alpha_f(x_0) = \sup\left\{ \beta \geq 0 \mid f \in C^{\beta}(x_0)  \right\}. 
\end{equation}  
Then, the Hausdorff dimension of the set of points with exponent $\alpha$, that is,  %of the sets $D_{\alpha} = \{  x \in \mathbb{R} \mid \alpha(x) = \alpha \}$,
\begin{equation}\label{SpectrumOfSingularitiesDef}
d(\alpha) = \operatorname{dim}_{\mathcal{H}}\{  x \in \mathbb{R} \mid \alpha_f(x) = \alpha \}, \qquad \forall \alpha \geq 0, 
\end{equation} 
 when regarded as a function of $\alpha$, is called the spectrum of singularities.
%where $\operatorname{dim}_{\mathcal{H}}$ stands for the Hausdorff dimension. 
This definition is usually extended to values of $\alpha$ yielding an empty set by setting their image to be $-\infty$. The spectrum of singularities is the principal object of study in multifractal analysis, and in fact a function is said to be multifractal if its spectrum of singularities is defined by \eqref{SpectrumOfSingularitiesDef} at least on an open interval of H\"older exponents $\alpha$.

Riemann's non-differentiable function was shown to be a multifractal by Jaffard \cite{Jaffard1996}, who proved
\begin{equation}\label{SpectrumOfSingularitiesJaffard}
d_R(\alpha) = \left\{
\begin{array}{cl}
4\alpha - 2, & \quad \text{if } \alpha \in [1/2,3/4], \\
0, & \quad \text{if } \alpha = 3/2, \\
-\infty, & \quad \text{otherwise}.
\end{array}
\right.
\end{equation}
The three functions $R$, $\phi_D$ and $\phi$ have the same regularity, so in fact \eqref{SpectrumOfSingularitiesJaffard} holds for all of them. 
With this result in hand, he also proved the validity of the Frisch-Parisi multifractal formalism  \cite{FrischParisi1985} for Riemann's function. To prove \eqref{SpectrumOfSingularitiesJaffard}, Jaffard established a relationship between the H\"older exponent of $R$ at an irrational point and a particular irrationality exponent of that irrational point that is related to the rate of convergence of its sequence of approximations by continued fractions. %This technique was also used later in \cite{KapitanskiRodnianski1999} to analyse the regularity of the solution of \eqref{FreeSchrodingerEquation} in the variable $s$ for every fixed $t$.

Multifractality is, thus, a concept measured in the domain of a function. With the geometric interpretation of Riemann's function in mind, a natural question is whether the multifractality of $\phi$ is translated from its domain to its image $\phi(\mathbb{R})$. We prove a partial result in this direction. 
\begin{thm}\label{TheoremHausdorffDimensionGeneralised}
Let $\phi$ be Riemann's non-differentiable function \eqref{Phi} and $D_{\sigma} = \{ x \in \mathbb{R} \mid \alpha_\phi(x) = \sigma \}$. Then,
\[  \operatorname{dim}_{\mathcal{H}} \phi(D_{\alpha}) \leq  \operatorname{dim}_{\mathcal{H}} \phi\Big(\bigcup_{\sigma \leq \alpha}D_{\sigma}\Big) \leq \frac{4\alpha - 2}{\alpha}, \qquad \forall \alpha \in [1/2,3/4]. \]
\end{thm}
Observe that according to \eqref{SpectrumOfSingularitiesJaffard}, the range $\alpha \in [1/2,3/4]$ is the only one of interest. Also, Theorem~\ref{TheoremHausdorffDimensionGeneralised} generalizes Theorem~\ref{TheoremHausdorffDimension} because the union $\cup_{\sigma \leq 3/4}D_{\sigma}$ covers the whole real line except a countable number of points, which are precisely those in $D_{3/2}$, the set of points where $\phi$ is differentiable. Then, the classical results of Hardy and Gerver imply that $D_{3/2} \subset \frac{1}{2\pi}\, \mathbb{Q}$, so $\operatorname{dim}_{\mathcal{H}} \phi(\mathbb{R}) = \operatorname{dim}_{\mathcal{H}} \phi(\cup_{\sigma \leq 3/4}D_{\sigma}) \leq 4/3$.

Taking the periodic property 
\begin{equation}\label{PeriodicityPhi}
 \phi\left(t+\frac{1}{2\pi}\right) = \phi(t) + \frac{i}{2\pi}, \qquad \qquad \forall t \in \mathbb{R}, 
\end{equation}
into account, Theorems~\ref{TheoremHausdorffDimension} and \ref{TheoremHausdorffDimensionGeneralised} can be proved using the asymptotic behavior of $\phi(t_x + h) - \phi(t_x)$ when $h \to 0$ for $x \in [0,1]$, where we denote $t_x = x/(2\pi)$. If $x = p/q \in \mathbb{Q}$ is an irreducible fraction, we will also write $t_{p/q} = t_{p,q}$. The proof of Theorem~\ref{TheoremHausdorffDimensionGeneralised} is also based on the classification of irrational points according to the rate of convergence of their approximations by continued fractions.

\subsection{Auxiliary geometric result: the asymptotic behavior of $\phi$ around rationals}

The asymptotic behavior of the original generalization $\phi_D$ of Riemann's function was computed by Duistermaat \cite{Duistermaat1991}. Thanks to it, he could explain the self-similar patterns of the graph of $R$ analytically. While one can get the asymptotic behavior of $\phi$ from Duistermaat's work using the relationship \eqref{FromPhiDuistermaatToPhi}, in this paper we will prove it directly. The reasons to do this are the following:
\begin{itemize}
	\item We do the computations from a different and, arguably, more intuitive perspective.
	\item Like in \cite{Duistermaat1991}, the main vehicle will be the relationship between the modular group and the Jacobi $\theta$ function, but this new approach allows to unravel the relationships with phenomena in other fields like Gauss sums in number theory and the Talbot effect in optics.
	\item To prove Theorems~\ref{TheoremHausdorffDimension} and \ref{TheoremHausdorffDimensionGeneralised} it is enough to work with the leading terms of the asymptotics, which can easily be deduced from Duistermaat's work. Even so, we compute the asymptotic behavior of $\phi$ so that machinery to prove future results is fully and explicitly available. It is the lower order terms which capture the self-similar properties of $\phi$, so they may be critical to tackle other geometric questions. For instance, it seems reasonable to think that they will be needed to obtain a lower bound for the Hausdorff dimension. They already proved to be vital in \cite{Eceizabarrena2020} to study the geometric differentiability of $\phi$.
\end{itemize}

For the sake of clarity, let us write here a simplified introductory version of the asymptotic behavior of $\phi$. It can be classified very cleanly, since the situation around any rational can be reduced to what happens around either 0 or 1/2. For the precise expressions I refer the reader to Propositions~\ref{thm:Asymptotic_At_0}, \ref{thm:Asymptotic_At_1_2}, \ref{thm:Asymptotic_At_Q013} and \ref{thm:Asymptotic_At_Q2}.

\begin{prop}\label{TheoremAsymptoticsSimplified}
Let $p,q \in \mathbb{Z}$ such that $ 0 \leq p < q$ and $\operatorname{gcd}(p,q)=1$. The asymptotic behavior of $\phi$ around the rational point $t_{p,q} = (p/q)/(2\pi)$ depends on  $q \pmod 4$ as follows:
\begin{itemize}
	\item The asymptotic behavior of $\phi$ around 0 is 
\begin{equation}
\phi(h) = \frac32\,  \frac{1+i}{\sqrt{2\pi}} \, \left( h^{1/2} - \frac{8\pi^2}{3}\, i \, \left[ \frac16 - 2\phi\left( \frac{-1}{16\pi^2h} \right) \right] h^{3/2} + O\left( h^{5/2}\right) \right).
\end{equation} 
and if $q \equiv 0,1,3 \pmod 4$, there exists an eighth root of unity $e_{p,q}$ such that 
\begin{equation}
\phi(t_{p,q} + h) - \phi(t_{p,q}) \approx \frac{e_{p,q}}{q^{3/2}} \, \phi(q^2\, h).
\end{equation}

\item  The asymptotic behavior of $\phi$ around $1/2$ is 
\begin{equation}
\phi(t_{1,2} + h) - \phi(t_{1,2}) = -16\,\frac{1-i}{\sqrt{2}}\,   \sum_{\substack{k=1 \\ k \text{ odd}}}^{\infty}{ \frac{e^{ik^2/(16h)}}{k^{2}} }   \,h^{3/2} + O\left( h^{5/2} \right),
\end{equation} 
and if $q \equiv 2 \pmod 4$, there exists an eighth root of unity $e_{p,q}$ such that 
\begin{equation}
\phi(t_{p,q} + h) - \phi(t_{p,q}) \approx \frac{e_{p,q}}{q^{3/2}}\, \left( \phi(t_{1,2} + q^2\, h) - \phi(t_{1,2}) \right).
\end{equation}
\end{itemize} 
\end{prop}

The second term in the asymptotic behavior around 0 captures the self-similar patterns of $\phi$ that can be identified in Figure~\ref{FIG_Curva}. Most importantly, this pattern appears around every rational $t_{p,q}$ with $q \equiv 0,1,3 \pmod 4$. This should play an important role to compute a lower bound for its Hausdorff dimension, but as already said, it is not needed to prove Theorems~\ref{TheoremHausdorffDimension} and \ref{TheoremHausdorffDimensionGeneralised}. Indeed, the following corollary with the leading order term is enough. 
\begin{cor}\label{thm:Global_Bound}
Let $p, q \in \mathbb{N}$ such that $\operatorname{gcd}(p,q) = 1$. Let also $M >0$. Then, there exists $C_M>0$ independent of $p$ and $q$ such that 
\begin{itemize}

	\item if $q \equiv 0,1,3 \pmod{4}$, 	
\begin{equation}
\left|  \phi(t_{p,q}+h) - \phi(t_{p,q}) \right|  \leq C_M \frac{|h|^{1/2}}{q^{1/2}}, \qquad \text{ whenever } |h| < \frac{M}{q^2}.
\end{equation}

	\item if  $q \equiv 2 \pmod{4}$, 
\begin{equation}
\left| \phi(t_{p,q}+h) - \phi(t_{p,q}) \right|  \leq C_M  q^{3/2}\, |h|^{3/2}, \qquad \text{ whenever } |h| \leq \frac{M}{q^2}. 
\end{equation}  

\end{itemize}
\end{cor}
This corollary corresponds to Corollaries~\ref{thm:Global_Bound_Q013} and \ref{thm:Global_Bound_Q2} in the main text.

%While these expressions allow to understand the main asymptotic behaviour of $\phi$ around rational numbers, its self-similar structure, already pointed out by Duistermaat \cite{Duistermaat1991} and vital in \cite{Eceizabarrena2019_Part2}, is hidden in the functions $Y$ and $Z$. 
%%Also, the fact that they are bounded functions with circle-like images makes the second term in the first asymptotic and the main term in the second asymptotic have a spiral-like pattern. 
%As already said, this will become clear in Propositions~\ref{PropAsymptoticQ013} and \ref{PropAsymptoticQ2}. 

\subsection{Discussion on a lower bound for the Hausdorff dimension}

The theorems in this paper are a first approach to the Hausdorff dimension of Riemann's non-differentiable function in its version shown in Figure~\ref{FIG_Curva}, which represents the trajectory of a polygonal vortex filament. Of course, the objective now turns into knowing whether the exact value of the dimension is precisely 4/3.
Some difficulties with respect to previous works are the following. First, dealing with Riemann's function is more complicated than working with Weierstrass' function due to its quadratic rather than exponential convergence. Also, Figure~\ref{FIG_Curva} is not a graph, so the control over the abscissa direction is lost. What is more, the set self-intersects many times, in a way that seems difficult to measure. Regarding self-similarity, unlike exactly self-similar fractals that have finitely many scaling laws, Figure~\ref{FIG_Curva} and specially the fact that the self-similar term in  Proposition~\ref{TheoremAsymptoticsSimplified} is multiplied by the continuously decreasing term $h^{3/2}$ suggest that $\phi$ may have a continuum of scaling laws.

There are some clues that vaguely suggest that the dimension might be 4/3, like the fact that the cover used in the proof of Theorem~\ref{TheoremHausdorffDimension} would no longer cover the set if the diameters are made slightly smaller and that the estimates used are sharp. A possible line of attack comes from deepening in the study of the multifractal setting of Theorem~\ref{TheoremHausdorffDimensionGeneralised}. In fact, analyzing the subsets $D_\alpha$ for a fixed $\alpha$ means studying irrationals with a fixed irrationality exponent, and this could be a way to isolate a set that has a single scaling, or at least a simpler scaling law.  

What we can be more convinced is the dimension being strictly greater than 1, due to the self-similar patterns already mentioned.
%Figure~\ref{FIG_Curva} and Proposition~\ref{TheoremAsymptoticsSimplified}. 
Even showing this would be an interesting contribution.

\subsection{Structure of the document}

Since Corollary~\ref{thm:Global_Bound} suffices to tackle Theorems~\ref{TheoremHausdorffDimension} and \ref{TheoremHausdorffDimensionGeneralised}, we begin by proving them in Section~\ref{Section_HausdorffDimension}. In Section~\ref{sec:Technical_Results} we prove some technical results corresponding to the multifractal setting of Theorem~\ref{TheoremHausdorffDimensionGeneralised}. In Section~\ref{sec:Asymptotics_Heuristics}, we explain the heuristics on how the asymptotic behavior of $\phi$ around rationals can be reduced to the asymptotic behavior around either 0 or $1/2$. We also explain how such reduction is deeply related to Gauss sums and the Talbot effect. Then, in Section~\ref{Section_BaseCases} we compute the asymptotics around 0 and $1/2$, and in Section~\ref{Section_Rationals} we compute the asymptotics around rationals by making the already mentioned reduction rigorous.

%Following what was explained right above, in section 2 we analyse the connection between Riemann's function and Gauss sums through the Talbot effect. We iteratively deduce the existence of a $\theta$-modular transformation that allows to compute the behaviour of $\phi$ around any $t_{p,q}$ using the behaviour around 0 and $t_{1,2}$.
%% which we may consider to be the \textit{base cases}.   
%
%In section 3, we compute the explicit $\theta$-modular transformation and we formally compute the asymptotic behaviour so that we have an overall perspective. 
%
%In sections afterwards, we develop the rigorous asymptotic behaviour. Sections~\ref{Section_AsymptoticIn0} and \ref{Section_AsymptoticIn12} are devoted to computations around 0 and  $t_{1,2}$. In sections~\ref{Section_AsymptoticInQ013} and \ref{Section_AsymptoticInQ2} we compute the asymptotic behaviour in the rest of the rationals using the transformations found in section~\ref{Section_LookingForTransformations}. Finally, in section~\ref{Section_HausdorffDimension} we prove Theorems~\ref{TheoremHausdorffDimension} and \ref{TheoremHausdorffDimensionGeneralised} on the Hausdorff dimension of $\phi(\mathbb{R})$.

\begin{acknowledgements*}

Special thanks to Luis Vega. I would also like to thank Fernando Chamizo, Albert Mas and Xavier Tolsa for interesting and useful discussions. 

The bulk of this work was developed while I was working at BCAM - Basque Center for Applied Mathematics. 

This research is supported by the Ministry of Education, Culture and Sport (Spain) under grant  FPU15/03078 - Formaci\'on de Profesorado Universitario, by the ERCEA under the Advanced Grant 2014 669689 - HADE and also by the Basque Government through the BERC 2018-2021 program and by the Ministry of Science, Innovation and Universities: BCAM Severo Ochoa accreditation SEV-2017-0718. It is also supported by the Simons Foundation Collaboration Grant on Wave Turbulence (Nahmod’s Award ID 651469).
\end{acknowledgements*}

\section{The Hausdorff dimension}\label{Section_HausdorffDimension}

In this last section, we prove Theorems~\ref{TheoremHausdorffDimension} and \ref{TheoremHausdorffDimensionGeneralised} based on Corollary~\ref{thm:Global_Bound}. Before going into the proofs, we recall that given $d \geq 0$, the $d$-Hausdorff content of diameter $\delta>0$ of a set $A \subset \mathbb{R}^n$ is 
\begin{equation}\label{HausdorffContent}
\mathcal{H}_{\delta}^d(A) = \inf\left\{ \sum_{i \in I}\left(\operatorname{diam}U_i\right)^d \mid A \subset \bigcup_{i \in I}{U_i}, \quad \operatorname{diam}U_i < \delta\quad \forall i \in I, \quad I \text{ countable} \right\},
\end{equation}
where the sets $U_i$ can be chosen to be open if needed. This is a decreasing function of $\delta$, and taking the limit $\delta \to 0$ yields the $d$-Hausdorff measure of $A$,
\begin{equation}\label{HausdorffMeasure}
\mathcal{H}^d(A) = \lim_{\delta \to 0}\mathcal{H}_{\delta}^d(A) = \sup_{\delta > 0}\mathcal{H}_{\delta}^d(A) .
\end{equation}
Finally, the Hausdorff dimension of $A$ is 
\begin{equation}\label{HausdorffDimension}
\operatorname{dim}_{\mathcal{H}}A = \inf\{\,d\, \mid \, \mathcal{H}^d(A) = 0\} = \sup\{\,d\, \mid \, \mathcal{H}^d(A) = \infty\}.
\end{equation}

\subsection{Proof of Theorem~\ref{TheoremHausdorffDimension}}\label{Subsection_TheoremHausdorffDimension}

The lower bound of Theorem~\ref{TheoremHausdorffDimension} is just a consequence of $\phi$ being a continuous and non-constant curve. Indeed, there exist $s,t \in \mathbb{R}$ with $s < t$ such that $\phi(s) \neq \phi(t)$.  Let $[\phi(s),\phi(t)] \subset \mathbb{R}^2$ denote the line segment connecting $\phi(s)$ and $\phi(t)$, and $L$ its infinite extension. Then, the orthogonal projection $P_{\perp}:\phi([s,t]) \to L$ is a Lipschitz map, so 
\begin{equation}
\operatorname{dim}_{\mathcal{H}}P_{\perp}\phi([s,t]) \leq \operatorname{dim}_{\mathcal{H}}\phi([s,t]),
\end{equation} 
see, for instance, \cite[Proposition 3.3]{Falconer2014}.
Since the continuity of $\phi$ implies $[\phi(s), \phi(t)] \subset P_{\perp}(\phi([s,t]))$, we get 
\begin{equation}
 \operatorname{dim}_{\mathcal{H}}\phi([0,1/(2\pi)]) \geq \operatorname{dim}_{\mathcal{H}}\phi([s, t]) \geq \operatorname{dim}_{\mathcal{H}}P_{\perp}\phi([s,t]) \geq \operatorname{dim}_{\mathcal{H}} [\phi(s), \phi(t)] = 1. 
\end{equation} 

Regarding the upper bound, it is enough to work with the set $\phi(\frac{1}{2\pi}\left((0,1)\cap\mathbb{I}\right))$, where $\mathbb{I}$ stands for the set of irrational numbers. This is because the periodic property \eqref{PeriodicityPhi} implies
\[  \phi(\mathbb{R}) = \bigcup_{k \in \mathbb{Z}}{\phi([k,k+1/2\pi])} = \bigcup_{k\in\mathbb{Z}}{ \left( \phi([0,1/2\pi]) + \frac{i}{2\pi}\, k \right) },  \]
and since the Hausdorff dimension of a countable union of sets is the supremum among the Hausdorff dimensions of each of the sets (see, for instance, \cite[Chapter 4]{Mattila1995}), we have
\begin{equation}
\operatorname{dim}_{\mathcal{H}}\phi(\mathbb{R}) = \sup_{k \in \mathbb{Z}} \operatorname{dim}_{\mathcal{H}} \left( \phi([0,1/2\pi]) + \frac{i}{2\pi}\, k \right).
\end{equation}
Of course, all such sets have the same Hausdorff dimension, so it is enough to work with, say, $k=0$. On the other hand, the set of rational points is countable and has therefore $\mathcal{H}^d$-measure zero for every $d >0$. Thus, $\phi([0,1/2\pi])$ has the same $\mathcal{H}^d$-measure as $\phi(\frac{1}{2\pi}\left((0,1)\cap\mathbb{I}\right))$. As a consequence, $\operatorname{dim}_{\mathcal{H}}\phi(\mathbb{R}) = \operatorname{dim}_{\mathcal{H}}\phi(\mathcal{I})$, where $\mathcal{I} = \frac{1}{2\pi}\left( (0,1) \cap \mathbb{I} \right)$.

It will be enough to find a proper countable cover of the set $\phi(\mathcal{I})$. First, we see that
\begin{equation}\label{CoverOfIrrationals}
  (0,1) \cap \mathbb{I} \, \, \,  \subset \bigcup_{\substack{1 \leq p < q \\ \operatorname{gcd}(p,q)=1 \\ q \geq Q_0}}{ B\left( \frac{p}{q}, \frac{1}{q^2} \right)  }, \qquad \forall Q_0 \in \mathbb{N}. 
\end{equation}
This cover is a direct consequence of the theory of continued fractions. Let $\rho \in (0,1) \cap \mathbb{I}$ and $\rho_n = p_n/q_n$ be its convergents by continued fractions for all $n \in \mathbb{N}$. These convergents are irreducible rationals such that $\lim_{n \to \infty}{q_n}=+\infty$ and $|\rho - p_n/q_n| < q_n^{-2}$ for every $n \in \mathbb{N}$. Consequently, for no matter how large $Q_0 \in \mathbb{N}$, we can find $N_0 \in \mathbb{N}$ such that $q_n \geq Q_0$ and  $|\rho - p_n/q_n| < q_n^{-2}$ for every $n > N_0$, hence \eqref{CoverOfIrrationals}.

Let now the asymptotics in Corollary~\ref{thm:Global_Bound} with $p = p_n$ and $q=q_n$ be evaluated at $h = h_n = t_\rho - t_{p_n,q_n}$ such that $t_{p_n,q_n} + h_n = t_{\rho}$. Then, $|h_n| < 1/(2\pi q_n^2)$, which implies $q_n^{3/2}|h_n|^{3/2} < q_n^{-1/2}|h_n|^{1/2}$. Thus, there exists $C>0$ such that 
\begin{equation}\label{UpperBoundSimplified}
|\phi(t_{\rho}) - \phi(t_{p_n,q_n})| \leq C\,\frac{|h_n|^{1/2}}{q_n^{1/2}} < \frac{C}{q_n^{3/2}}, \qquad \forall n \in \mathbb{N}. 
\end{equation} 
Thus, \eqref{CoverOfIrrationals} is translated to the image of $\phi$ because \eqref{UpperBoundSimplified} shows that
\begin{equation}\label{CoverOfirrationalsPhi}
  \phi ( \mathcal{I} ) \subset \bigcup_{\substack{1 \leq p < q \\ \operatorname{gcd}(p,q)=1 \\ q \geq Q_0}}{ B\left( \phi\left(t_{p,q}\right), \frac{C}{q^{3/2}} \right)  }, \qquad \forall Q_0 \in \mathbb{N}.
\end{equation}
Let $d >0$. This cover for $\phi(\mathcal I )$ yields an upper bound of the $d$-Hausdorff content \eqref{HausdorffContent} of diameter $\delta <C/Q_0^{3/2}$, since we have
\begin{equation}\label{eq:Estimation_Of_The_Hausdorff_Content}
\mathcal{H}^d_{C/Q_0^{3/2}}(\phi(\mathcal{I}))  \leq  \sum_{\substack{1 \leq p < q \\ \operatorname{gcd}(p,q)=1 \\ q \geq Q_0}}{ \left( \operatorname{diam} B\left( \phi\left(t_{p,q}\right), \frac{C}{q^{3/2}} \right) \right)^d } = C^d\,\sum_{ q = Q_0}^{\infty}{ \frac{\varphi(q)}{q^{3d/2}} }   \leq C^d\,\sum_{ q = Q_0}^{\infty}{ \frac{1}{q^{3d/2 - 1}} } 
\end{equation}
for every $Q_0 \in \mathbb{N}$.  Here, $\varphi$ is Euler's totient function, whose trivial but in general best bound $\varphi(q) < q$ we used above. Then, take the limit $Q_0 \to \infty$ so that
\begin{equation}
\mathcal{H}^d(\phi(\mathcal{I})) = \lim_{Q_0 \to \infty} \mathcal{H}^d_{C/Q_0^{3/2}}(\phi(\mathcal{I})) \leq  C^d\, \lim_{Q_0 \to \infty} \sum_{ q = Q_0}^{\infty}{ \frac{1}{q^{3d/2 - 1}} } .
\end{equation}
The sum inside the limit converges if and only if $3d/2 - 1 > 1$, or equivalently is and only if $d > 4/3$, so
\begin{equation}
\mathcal{H}^d(\phi(\mathcal{I})) = 0, \qquad \forall d > 4/3.
\end{equation}
According to the definition of the Hausdorff dimension \eqref{HausdorffDimension},  this implies $ \operatorname{dim}_{\mathcal{H}}( \phi(\mathcal{I} ))  \leq4/3 $.
\qed

\begin{rem}
Using the Dirichlet approximation theorem instead of the continued fraction theory to obtain a cover like \eqref{CoverOfIrrationals} gives some extra information about $\mathcal{H}^{4/3}(\phi(\mathcal{I}))$. Dirichlet's theorem states that given a natural number $N \in \mathbb{N}$ and any irrational $\rho$, there exist $p,q \in \mathbb{Z}$, such that $1 \leq q \leq N$ and $|\rho - p/q| < 1/(qN)$. This implies that 
\begin{equation}
(0,1) \cap \mathbb I \,\,\, \subset \bigcup_{\substack{ 1 \leq  q \leq N \\ 1 \leq p \leq q }} B \left(\frac{p}{q}, \frac{1}{qN}\right), \qquad \forall N \in \mathbb{N}.
\end{equation}
Fix $N \in \mathbb N$ and let $p_N(\rho)/q_N(\rho)$ be the approximation of $\rho$ corresponding to $N$. Plug this in \eqref{UpperBoundSimplified} so that we get
\begin{equation}
\left| \phi(t_\rho) - \phi(t_{p_N(\rho),q_N(\rho)}) \right|   \leq C\, \frac{ |\rho - p_N(\rho)/q_N(\rho)|^{1/2} }{ q_N(\rho)^{1/2} } \leq  \frac{C}{ q_N(\rho)\, N^{1/2}  },
\end{equation}
which means that 
\begin{equation}
\phi(\mathcal{I}) \subset \bigcup_{\substack{ 1 \leq q \leq N \\ 1 \leq p \leq q }} B \left( \phi(t_{p,q}), \frac{C}{q N^{1/2}} \right), \qquad \forall N \in \mathbb N. 
\end{equation}
Moreover, the diameters of the balls satisfy $ 1/(qN^{1/2}) \leq N^{-1/2}$. Thus, for $1 \leq d < 2$, 
\begin{equation}
\mathcal{H}^d_{N^{-1/2}}(\phi(\mathcal I)) \leq \sum_{\substack{ 1 \leq q \leq N \\ 1 \leq p \leq q }} \frac{C^d}{(qN^{1/2})^d} = \frac{C^d}{N^{d/2}}\,\sum_{ q=1 }^N \frac{1}{q^{d-1}} \leq C^d\,  \frac{N^{2- 3d/2}}{(2-d)},
\end{equation}
which shows as before that $\mathcal{H}^d(\phi(\mathcal I)) = \lim_{N\to \infty}  \mathcal{H}^d_{N^{-1/2}}(\phi(\mathcal I)) = 0$ for every $d > 4/3$, but also and more interestingly, 
\begin{equation}
\mathcal{H}^{4/3}(\phi(\mathcal I)) = \lim_{N\to \infty}  \mathcal{H}^{4/3}_{N^{-1/2}}(\phi(\mathcal I)) \leq \frac{C^d}{2 - 4/3} = \frac32 C^d.
\end{equation}
\end{rem}

\subsection{Proof of Theorem~\ref{TheoremHausdorffDimensionGeneralised}}\label{Subsection_TheoremHausdorffDimensionGeneralised}
We follow the structure of the proof of Theorem~\ref{TheoremHausdorffDimension}, but we use deeper results that relate the rate of convergence of the approximations by continued fractions with the H\"older regularity coefficients defined in \eqref{MaximalHolderRegularity}. 

Let $p_n/q_n$ be the $n$-th covergent by continued fractions of $\rho \in (0,1) \cap \mathbb{I}$. As above, $| \rho - p_n/q_n | < q_n^{-2}$, but now we want to quantify how smaller than $q_n^{-2}$ this error is. For that, define the sequence $(\gamma_n)_{n \in \mathbb{N}}$ as
\begin{equation}\label{DefinitionOfGammaN}
\left| \rho - \frac{p_n}{q_n} \right| = \frac{1}{q_n^{\gamma_n}} , \qquad \forall n \in \mathbb{N}.
\end{equation}  
It is clear that $\gamma_n > 2$ for every $n \in \mathbb{N}$. Of all convergents, let us work only with the approximations with $q_n \equiv 0,1,3 \pmod{4}$, which are always infinitely many (see Lemma~\ref{lem:Continued_Fraction_Auxiliary}), and define 
\begin{equation}\label{DefinitionOfGamma}
\begin{split}
\gamma(\rho) & = \sup\left\{ \tau \mid \gamma_n \geq \tau \text{ for infinitely many } n \in \mathbb{N} \text{ such that } q_n \equiv 0,1,3 \,(\text{mod } 4) \right\} \\
&  = \limsup_{\substack{n \to \infty \\ q_n \equiv 0,1,3 \,(\text{mod } 4)}} \gamma_n, 
\end{split}
\end{equation} 
There is a direct connection between $\gamma$ and the H\"older exponent $\alpha_\phi$ \eqref{MaximalHolderRegularity} given by
\begin{equation}\label{CorrespondenceJaffard}
\alpha_\phi(t_{\rho})  = \frac12 + \frac{1}{2\gamma(\rho)}. 
\end{equation} 
This identity is an adaptation of the original result for $\phi_D$ shown by Jaffard in \cite{Jaffard1996}, see Subsection~\ref{sec:CorrespondenceJaffard} for details and proof. 

The idea of the proof is that the definition of $\gamma_n$ allows to improve the bound in \eqref{UpperBoundSimplified} because now, if $h_n = t_\rho - t_{p_n,q_n}$, 
\begin{equation}\label{eq:Estimations_With_Gamma_n}
|\phi(t_{\rho}) - \phi(t_{p_n,q_n})| \leq C\,\frac{|h_n|^{1/2}}{q_n^{1/2}} = \frac{C}{q_n^{(1+\gamma_n)/2}} , \qquad \forall n \in \mathbb{N},
\end{equation}
which is smaller than $C/q_n^{3/2}$, and then $\gamma(\rho)$ can be used to control the exponent $(1+\gamma_n)/2$. Thus, we take the set of points with fixed $\gamma(\rho)=\gamma$ and we cover it like in \eqref{CoverOfirrationalsPhi} but with balls of smaller diameter, yielding a better estimation for the Hausdorff dimension. Finally, the correspondence \eqref{CorrespondenceJaffard} connects these sets with the sets where $\phi$ has a given regularity. 

 Define the sets of points of a determinate coefficient $\beta \geq 2$,
\begin{equation}\label{CorrespondenceBetweenSets}
R_{\beta}  = \left\{ t_\rho \in \mathcal I  \mid \gamma(\rho) = \beta  \right\} = D_{\frac12 + \frac{1}{2\beta}} \cap \mathcal{I},
\end{equation} 
where $t_\rho = \rho/2\pi$, $D_\sigma = \{ x \in \mathbb{R} \, \mid \, \alpha(x) = \sigma  \}$ and the last equality holds because of \eqref{CorrespondenceJaffard}. Let $\beta > 2$ and $t_{\rho} \in \cup_{\sigma \geq \beta}R_{\sigma}$ so that $\gamma(\rho) \geq \beta$. Then, choose $\epsilon >0$ such that $\gamma(\rho) - \epsilon \geq \beta - \epsilon > 2$. By definition of $\gamma(\rho)$, the set of indices 
\begin{equation}
A_{\rho, \epsilon} = \{ n \in \mathbb{N} \mid q_n \equiv 0,1,3\, (\text{mod } 4) \text{ and } \gamma_n > \beta - \epsilon \} 
\end{equation} 
is infinite for all $\epsilon > 0$ as above, and hence, from \eqref{eq:Estimations_With_Gamma_n} we get
\begin{equation}
 |\phi(t_{\rho}) - \phi(t_{p_n,q_n})|  <  \frac{C}{q_n^{(1+\beta - \epsilon)/2}} , \qquad \forall n \in A_{\rho,\epsilon}.  
\end{equation} 
As in \eqref{CoverOfirrationalsPhi}, this shows that
\[  \phi \Big(  \bigcup_{\sigma \geq \beta}R_{\sigma}  \Big) \subset \bigcup_{\substack{ 1 \leq p < q \\ \operatorname{gcd}(p,q) = 1 \\ q \geq Q_0 }}B\left(  \phi(t_{p,q}),\frac{C}{q^{(1+\beta - \epsilon)/2}} \right), \qquad \forall Q_0 \in \mathbb{N}.  \]
Repeating the same procedure as in \eqref{eq:Estimation_Of_The_Hausdorff_Content}, we get
\[ \mathcal{H}^d \left( \phi \Big(  \bigcup_{\sigma \geq \beta}R_{\sigma}  \Big) \right) \leq C^d  \lim_{Q_0 \to \infty}\sum_{q=Q_0}^{\infty}{  \frac{1}{q^{\frac{1+\beta - \epsilon}{2}d - 1}} } = 0, \qquad \forall d > \frac{4}{1+\beta - \epsilon},  \]
so $\operatorname{dim}_{\mathcal{H}}\phi \Big(  \bigcup_{\sigma \geq \beta}R_{\sigma}  \Big) \leq d$ for every $d > 4/(1+\beta-\epsilon)$ and every $\epsilon >0$. Since this is valid for every $0 < \epsilon < \beta - 2$, let $\epsilon \to 0$ to we conclude that 
\[  \operatorname{dim}_{\mathcal{H}}\phi \Big( \bigcup_{\sigma \geq \beta}R_{\sigma}  \Big) \leq \frac{4}{1+\beta}, \qquad \forall \beta > 2.  \]
By the correspondences \eqref{CorrespondenceJaffard} and \eqref{CorrespondenceBetweenSets}, we get the result for the H\"older regularity sets,
\[  \operatorname{dim}_{\mathcal{H}}\phi \Big( \mathcal{I} \cap \bigcup_{\sigma \leq \alpha}D_{\sigma}  \Big) \leq \frac{4\alpha-2}{\alpha},  \qquad \text{ for every } \quad \frac12 \leq \alpha < \frac34.    \]
This is also valid for $\alpha = 3/4$. Indeed, every irrational $\rho$ satisfies $\gamma(\rho) \geq 2$, which according to \eqref{CorrespondenceJaffard} means that $\alpha(t_\rho) \leq 3/4$. This means that all the irrational $t_\rho$ are in $\mathcal{I} \cap \bigcup_{\sigma \leq 3/4}D_{\sigma}$, so that the difference with the whole interval  $ [0,1/(2\pi)]$ is a subset of the rationals $\{ t_x \mid x \in \mathbb{Q} \cap [0,1] \}$, at most a countable set which has Hausdorff dimension 0. Hence, according to Theorem~\ref{TheoremHausdorffDimension}, 
\begin{equation}
\operatorname{dim}_{\mathcal{H}}\phi \Big( \mathcal{I} \cap \bigcup_{\sigma \leq 3/4}D_{\sigma}  \Big)  = \operatorname{dim}_{\mathcal{H}} \phi([0,1/(2\pi)]) \leq 4/3.
\end{equation}
Like in the proof of Theorem~\ref{TheoremHausdorffDimension}, the theorem follows because the periodic property \eqref{PeriodicityPhi} implies that  $\phi\left( \bigcup_{\sigma \leq \alpha}D_{\sigma}\right)$ is a countable union of translates of $\phi \left( \mathcal{I} \cap \bigcup_{\sigma \leq \alpha}D_{\sigma}  \right)$. Also, the first inequality of the theorem is just a consequence of the inclusion $ D_{\alpha} \subset \bigcup_{\sigma \leq \alpha}D_{\sigma} $.

\section{Technical results for Section~\ref{Section_HausdorffDimension}}\label{sec:Technical_Results}

\subsection{Proof of the correspondence (\ref{CorrespondenceJaffard})}\label{sec:CorrespondenceJaffard}

In \cite{Jaffard1996}, Jaffard proved 
\begin{equation}\label{eq:Connection_Alpha_Tau_Proof}
\alpha_{\phi_D}(\rho) = \frac12  + \frac{1}{2\tau(\rho)},
\end{equation}
where $\phi_D$ is Duistermaat's version \eqref{PhiDuistermaat}, $\alpha_{\phi_D}$ is the H\"older exponent of $\phi_D$ defined in \eqref{MaximalHolderRegularity} and 
\begin{equation}\label{eq:Definition_Of_Tau_Proof}
\tau(x) = \sup \left\{ \tau \, : \, \Big| x - \frac{p_n}{q_n} \Big| < \frac{1}{q_n^\tau}, \, \text{ for infinitely many } \frac{p_n}{q_n} \text{ such that not both } p_n,q_n \text{ are odd}\,   \right\},
\end{equation}
which is similar to the irrationality exponent of $\rho$ \footnotemark \footnotetext{The irrationality exponent of an irrational $\rho$ is defined as 
\begin{equation}
\mu(\rho) = \sup \left\{ \mu >0 \, : \, \Big|  \rho - \frac{p}{q}  \Big| < \frac{1}{q^\mu} \text{ for infinitely many rationals } \frac{p}{q} \,  \right\},
\end{equation}
and it can be proved (as in Lemma~\ref{thm:Lemma_Gamma_Tau_Aux}) that equivalently, if $p_n/q_n$ are the convergents of $\rho$, then
\begin{equation}
\mu(\rho) = \sup \left\{ \mu >0 \, : \, \Big|  \rho - \frac{p_n}{q_n}  \Big| < \frac{1}{q^\mu} \text{ for infinitely many } n \in \mathbb N \,  \right\}.
\end{equation}
}. 
In this subsection, we check that \eqref{CorrespondenceJaffard} is the equivalent expression for $\phi$, where $\tau$ is replaced by $\gamma$ \eqref{DefinitionOfGamma}. 

It is clear from \eqref{FromPhiDuistermaatToPhi} that  $\phi_D$ and $\phi$ share regularity properties. More precisely, $\phi$ has at $t_\rho = \rho/2\pi$ the regularity that $\phi_D$ has at $2\rho$, so
\begin{equation}
\alpha_{\phi}(t_{\rho}) = \alpha_{\phi_D}(2\rho).
\end{equation}
Therefore, from \eqref{eq:Connection_Alpha_Tau_Proof} we immediately get
\begin{equation}
\alpha_{\phi}(t_\rho) = \frac12 + \frac{1}{2\tau(2\rho)}.
\end{equation}
However, we want to connect $\alpha_{\phi}(t_\rho)$ directly with some irrationality exponent of $\rho$, not of $2\rho$. It is usual in this transition (see Section~\ref{Section_Heuristics}, \eqref{eq:Classification_Of_Rationals}) that the condition of $p_n,q_n$ not being both odd for $\phi_D$ turns into $q_n \equiv 0,1,3 \pmod{4}$ for $\phi$, so we expect the correct exponent to be
\begin{equation}\label{eq:Definition_Of_Gamma_Proof}
\gamma(x) = \sup\left\{ \gamma \, : \, \Big| x - \frac{p_n}{q_n} \Big| < \frac{1}{q_n^\gamma} \text{ for infinitely many } n \in \mathbb{N} \text{ with } q_n \equiv 0,1,3 \,(\text{mod } 4) \right\},
\end{equation}
which is the same as \eqref{DefinitionOfGamma}. We prove the following: 
\begin{lem}\label{thm:Lemma_Gamma_Tau}
Let $x \in \mathbb R \setminus \mathbb Q$. Then, $\gamma(x) = \tau(2x)$. 
\end{lem}
%Although this lemma can be proved directly with the definitions of $\gamma$ and $\tau$ above and using the convergents of $\rho$ and $2\rho$,
We split the proof in two steps. First, we prove in Lemma~\ref{thm:Lemma_Gamma_Tau_Aux} that $\gamma$ and $\tau$ can be defined using any rational, that is, by
\begin{equation}\label{eq:Definition_Of_Tau_R}
\tau_R(x) = \sup \left\{ \tau \, : \, \Big| x - \frac{p}{q} \Big| < \frac{1}{q^\tau}, \, \text{ for infinitely many rationals } \frac{p}{q} \text{ not both odd}\,   \right\},
\end{equation}
and 
\begin{equation}\label{eq:Definition_Of_Gamma_R}
\gamma_R(x) = \sup\left\{ \gamma : \Big| x - \frac{p}{q} \Big| < \frac{1}{q^\gamma} \text{ for infinitely many rationals } \frac{p}{q}  \text{ with } q \equiv 0,1,3 \,(\text{mod } 4) \right\},
\end{equation}
where in both definitions all fractions must be irreducible. Then, we prove the equality of \eqref{eq:Definition_Of_Tau_R} and \eqref{eq:Definition_Of_Gamma_R} in Lemma~\ref{thm:Lemma_Gamma_Tau_R}. 

\begin{lem}\label{thm:Lemma_Gamma_Tau_Aux}
Let $x \in \mathbb{R} \setminus \mathbb{Q}$. Then, $\tau_R(x) = \tau(x)$ and $\gamma_R(x) = \gamma(x)$. 
\end{lem}
\begin{proof}
We prove $\tau_R(x) = \tau(x)$, the proof for $\gamma$ is analogous. First, it is clear that 
\begin{equation}\label{GoodApproximation}
\begin{split}
& \left\{ \tau \, : \, \Big| x - \frac{p_n}{q_n} \Big| < \frac{1}{q_n^\tau} \, \text{ for infinitely many  convergents } \frac{p_n}{q_n} \text{ not both odd}\,   \right\} \\
& \qquad \subset \left\{ \tau \, : \, \Big| x - \frac{p}{q} \Big| < \frac{1}{q^\tau} \, \text{ for infinitely many rationals } \frac{p}{q} \text{ not both odd}\,   \right\},
\end{split}
\end{equation}
so taking the supremum we get $\tau(x) \leq \tau_R(x)$. Let now $\tau$ such that there are infinitely many rationals $p/q$ such that $p$ and $q$ are not both odd and $|x-p/q| < q^{-\tau}$. Assume that $\tau > 2$, so that 
\begin{equation}
\Big| x - \frac{p}{q} \Big| < \frac{1}{q^\tau} \leq \frac{1}{2q^2} \quad \Longleftrightarrow \quad 2 \leq q^{\tau - 2}
\end{equation}
holds whenever $q > 2^{1/(\tau-2)}$. Since we are working with infinitely many rationals $p/q$, in particular infinitely many of them satisfy this last property. It is a property of continued fractions (see \cite[Theorem 19]{Khinchin1964}) that every approximation satisfying the left hand side of \eqref{GoodApproximation} is a convergent of $x$, so there are infinitely many continued fraction convergents $p_n/q_n$ such that $|\rho - p_n/q_n| < q_n^{-\tau}$. Thus, 
\begin{equation}\label{eq:Part_In_Proof_Tau_Aux}
\begin{split}
& \left\{ \tau > 2 \, : \, \Big| x - \frac{p}{q} \Big| < \frac{1}{q^\tau} \, \text{ for infinitely many rationals } \frac{p}{q} \text{ not both odd}\,   \right\} \\
& \qquad \subset \left\{ \tau > 2 \, : \, \Big| x - \frac{p_n}{q_n} \Big| < \frac{1}{q_n^\tau} \, \text{ for infinitely many convergents } \frac{p_n}{q_n} \text{ not both odd}\,   \right\}.
\end{split}
\end{equation}

To continue, we need to check that $ \tau(x) \geq 2$. This is a consequence of $|x - p_n/q_n| < q_n^{-2}$ being true for all $n \in \mathbb N$ and the fact that there are infinitely many convergents $p_n/q_n$ with not both $p_n$ and $q_n$ odd (in fact, consecutive convergents $p_n, q_n, p_{n-1},q_{n-1}$ cannot all be odd because $q_np_{n-1} - q_{n-1}p_n = (-1)^n$, see \cite[Theorem 2]{Khinchin1964}). 

Now, the trivial inequality we proved in the beginning of the proof implies that $2 \leq \tau(x) \leq \tau_R(x)$. Thus, we separate two cases. If $\tau_R(x)=2$, then $2 \leq \tau(x) \leq \tau_R(x) = 2$ and hence $\tau(x) = \tau_R(x)$. Otherwise, $\tau_R(x) > 2$, and  by the definition of the supremum and by \eqref{eq:Part_In_Proof_Tau_Aux},
\begin{equation}
\begin{split}
\tau_R(x)  & = \sup \left\{ \tau > 2 \, : \, \Big| x - \frac{p}{q} \Big| < \frac{1}{q^\tau} \, \text{ for infinitely many rationals } \frac{p}{q} \text{ not both odd}\,   \right\} \\
&  \leq \sup \left\{ \tau > 2 \, : \, \Big| x - \frac{p_n}{q_n} \Big| < \frac{1}{q_n^\tau} \, \text{ for infinitely many convergents } \frac{p_n}{q_n} \text{ not both odd}\,   \right\} \\
&  \leq \sup \left\{ \tau \geq  2 \, : \, \Big| x - \frac{p_n}{q_n} \Big| < \frac{1}{q_n^\tau} \, \text{ for infinitely many convergents } \frac{p_n}{q_n} \text{ not both odd}\,   \right\} \\
&  = \tau(\rho),
\end{split}
\end{equation}
and the proof is complete.
\end{proof}

Thanks to Lemma~\ref{thm:Lemma_Gamma_Tau_Aux}, Lemma~\ref{thm:Lemma_Gamma_Tau} follows from the following.
\begin{lem}\label{thm:Lemma_Gamma_Tau_R}
Let $x \in \mathbb R \setminus \mathbb Q $. Then, $\gamma_R(x) = \tau_R(2x)$. 
\end{lem}
\begin{proof}
Rewrite $\tau_R(2x)$ as
\begin{equation}
\begin{split}
\tau_R(2x) & =  \sup \left\{ \tau \, : \, \Big| 2x - \frac{p}{q} \Big| < \frac{1}{q^\tau}, \, \text{ for infinitely many } \frac{p}{q} \text{ not both odd}\,   \right\} \\
& =  \sup \left\{ \tau \, : \, \Big| x - \frac{1}{2}\frac{p}{q} \Big| < \frac{1}{2q^\tau}, \, \text{ for infinitely many } \frac{p}{q} \text{ not both odd}\,   \right\}.
\end{split}
\end{equation}
We want to write the bound $1/(2q^\tau)$ in terms of the denominator of the new fraction $p/(2q)$, and there are two different cases:
\begin{enumerate}
	\item If $p$ is even and $q$ is odd, then $p/(2q) = (p/2)/q$, and the denominator is $q$. We let the condition as $|x - (p/2)/q| < 1/(2q^\tau)$.
	\item If $p$ is odd and $q$ is even, then $p/(2q)$, and the denominator is $2q$. 	We rewrite the condition as $|x - p/(2q)| < 2^{\tau - 1}/(2q)^\tau$.
\end{enumerate}
The condition must hold for infinitely many rationals, so if we  relabel as
\begin{equation}\label{eq:P1}
\tag{$P1_\tau$}
\Big| x - \frac{p}{q} \Big| < \frac{1}{2q^\tau}, \qquad \text{ if } q \text{ odd }
\end{equation}
and 
\begin{equation}\label{eq:P2}
\tag{$P2_\tau$}
\Big| x - \frac{p}{q} \Big| < \frac{2^{\tau - 1}}{q^\tau}, \qquad \text{ if } q \equiv 0 \pmod{4},
\end{equation}
then $\tau_R(2x)$ is equivalently given by
\begin{equation}
\tau_R(2x) = \sup \left\{ \tau \, : \,  \text{ infinitely many } \frac{p}{q} \text{ satisfy their corresponding } \eqref{eq:P1} \text{ or } \eqref{eq:P2}  \right\},
\end{equation}
where the rationals have to be such that $q \equiv 0,1,3 \pmod{4}$. 

By Lemma~\ref{thm:Lemma_Gamma_Tau_Aux}, we know that $\tau_R(2x),\gamma_R(x) \geq 2$, so we may work only with $\tau, \gamma \geq 2$ all along the proof. Fix $\epsilon >0$.

With the definition of $\gamma_R(x)$ in mind, assume that $\gamma \geq 2$ is such that $|x - p/q| < 1/q^{\gamma+\epsilon}$ for infinitely many rationals with $q \equiv 0,1,3 \pmod{4}$. For the ones satisfying $q \equiv 0 \pmod 4$,  
\begin{equation}
\frac{1}{q^{\gamma + \epsilon}} < \frac{2}{q^\gamma} \leq \frac{2^{\gamma - 1}}{q^\gamma}
\end{equation}
always holds, so (\textcolor{blue}{$P2_\gamma$}) holds. Also, for those with $q \equiv 1,3 \pmod 4$,  
\begin{equation}
\frac{1}{q^{\gamma + \epsilon}} < \frac{1}{2q^\gamma} \, \Longleftrightarrow 2 < q^\epsilon,
\end{equation}
so (\textcolor{blue}{$P1_\gamma$}) holds for $q>2^{1/\epsilon}$. In short, all rationals that satisfy $q>2^{1/\epsilon}$, which are infinitely many, satisfy their corresponding (\textcolor{blue}{$P1_\gamma$}) or (\textcolor{blue}{$P2_\gamma$}), so 
\begin{equation}
\begin{split}
& \left\{ \gamma \geq 2 \mid \Big| x - \frac{p}{q} \Big| < \frac{1}{q^{\gamma+\epsilon}} \text{ for infinitely many } \frac{p}{q}  \text{ with } q \equiv 0,1,3 \,(\text{mod } 4) \right\} \\
& \qquad \quad \subset \left\{ \tau \geq 2 \, : \,  \text{ infinitely many } \frac{p}{q} \text{ satisfy } \eqref{eq:P1} \text{ or } \eqref{eq:P2}  \right\},
\end{split}
\end{equation}
or equivalently, 
\begin{equation}\label{Epsilons}
\begin{split}
& \left\{ \sigma \geq 2+\epsilon  \mid \Big| x - \frac{p}{q} \Big| < \frac{1}{q^{\sigma}} \text{ for infinitely many } \frac{p}{q}  \text{ with } q \equiv 0,1,3 \,(\text{mod } 4) \right\} - \epsilon \\
& \qquad \quad \subset \left\{ \tau \geq 2 \, : \,  \text{ infinitely many } \frac{p}{q} \text{ satisfy } \eqref{eq:P1} \text{ or } \eqref{eq:P2}  \right\}.
\end{split}
\end{equation}
Then, if we assume that $\gamma_R(x) > 2$ and choose $\epsilon < \gamma_R(x)-2$, then $\gamma_R(x) > 2+\epsilon$ and the supremum of the left hand side set of \eqref{Epsilons} is $\gamma_R(x)-\epsilon$. Then, taking supremums in \eqref{Epsilons}, we get
\begin{equation}\label{eq:Gamma_Less_Than_Tau}
\gamma_R(x) > 2 \quad \Longrightarrow \quad \gamma_R(x) - \epsilon \leq \tau_R(2x), \qquad \forall \epsilon < \gamma_R(x) - 2.
\end{equation}
This is one of the inequalities we need. In particular, $2 < \gamma_R(x) - \epsilon \leq \tau_R(2x)$. Thus,
\begin{equation}\label{eq:Big_Big_1}
\gamma_R(x) > 2 \quad \Longrightarrow \quad \tau_R(2x) > 2. 
\end{equation}

We look now for the reverse inequality. Let $\tau \geq 2$ and assume that there are infinitely many rationals satisfying their corresponding ($\textcolor{blue}{P1_{\tau+\epsilon}}$) or ($\textcolor{blue}{P2_{\tau+\epsilon}}$).
%their corresponding condition, but now with exponent $\tau + \epsilon$, so
%\begin{equation}\label{eq:P1_Epsilon}
%\tag{$P1_\epsilon$}
%\Big| \rho - \frac{p}{q} \Big| < \frac{1}{2q^{\tau+\epsilon}} \qquad \text{ if } q \text{ odd }
%\end{equation}
%and 
%\begin{equation}\label{eq:P2_Epsilon}
%\tag{$P2_\epsilon$}
%\Big| \rho - \frac{p}{q} \Big| < \frac{2^{\tau +\epsilon - 1}}{q^{\tau+\epsilon}} \qquad \text{ if } q \equiv 0 \pmod{4}
%\end{equation}
For the rationals satisfying ($\textcolor{blue}{P1_{\tau+\epsilon}}$), 
\begin{equation}
\Big| x - \frac{p}{q} \Big| < \frac{1}{2q^{\tau+\epsilon}} < \frac{1}{q^\tau}
\end{equation}
always holds, and for those satisfying ($\textcolor{blue}{P2_{\tau+\epsilon}}$), we have
\begin{equation}
\Big| x - \frac{p}{q} \Big| < \frac{2^{\tau +\epsilon - 1}}{q^{\tau+\epsilon}} < \frac{1}{q^\tau} \quad \Longleftrightarrow \quad 2^{\tau + \epsilon - 1} < q^\epsilon,
\end{equation}
which holds for all that satisfy $q > 2^{(\tau + \epsilon -1 )/\epsilon}$. We are working with an infinite set of rationals, so infinitely many of them satisfy $q > 2^{(\tau + \epsilon -1 )/\epsilon}$. Thus, infinitely many of them, all with $q \equiv 0,1,3 \pmod 4$, satisfy $|x - p/q| < 1/q^\tau$. Hence, 
\begin{equation}
\begin{split}
& \left\{ \tau \geq 2  \, : \,  \text{ infinitely many } \frac{p}{q} \text{ satisfy (} \textcolor{blue}{P1_{\tau+\epsilon}} \text{) or (} \textcolor{blue}{P2_{\tau+\epsilon}} \text{)}  \right\} \\
& \qquad \quad \subset \left\{ \gamma \geq 2 \, : \, \Big| x - \frac{p}{q} \Big| < \frac{1}{q^{\gamma}} \text{ for infinitely many } \frac{p}{q}  \text{ with } q \equiv 0,1,3 \,(\text{mod } 4) \right\},
\end{split}
\end{equation}
or equivalently, 
\begin{equation}\label{Epsilons2}
\begin{split}
& \left\{ \sigma \geq 2 + \epsilon  \, : \,  \text{ infinitely many } \frac{p}{q} \text{ satisfy (} \textcolor{blue}{P1_{\sigma}} \text{) or (} \textcolor{blue}{P2_{\sigma}} \text{)}  \right\} - \epsilon \\
& \qquad \quad \subset \left\{ \gamma \geq 2 \, : \, \Big| x - \frac{p}{q} \Big| < \frac{1}{q^{\gamma}} \text{ for infinitely many } \frac{p}{q}  \text{ with } q \equiv 0,1,3 \,(\text{mod } 4) \right\},
\end{split}
\end{equation}
As before, if we assume $\tau_R(2x) > 2$, then choose $\epsilon < \tau_R(2x) - 2$ so that $2+\epsilon < \tau_R(2x)$. This implies that the supremum of the set on the left hand side of \eqref{Epsilons2} is precisely $\tau_R(2x) - \epsilon$, so we get
\begin{equation}\label{eq:Tau_Less_Than_Gamma}
\tau_R(2x) > 2 \quad \Longrightarrow \quad \tau_R(2x) - \epsilon \leq \gamma_R(x), \qquad \forall \epsilon <  \tau_R(2x) - 2.
\end{equation}
In particular, $ 2 < \tau_R(2x) - \epsilon \leq \gamma_R(x)$, so we also get
\begin{equation}\label{eq:Big_Big_2}
\tau_R(2x) > 2 \quad \Longrightarrow \quad \gamma_R(x)>2.
\end{equation}

We are ready to conclude. Joining \eqref{eq:Big_Big_1} and \eqref{eq:Big_Big_2} gives
\begin{equation}
\gamma_R(x) = 2 \quad \Longleftrightarrow \quad \tau_R(2x) = 2.
\end{equation}
Also, when $\gamma_R(x), \tau_R(2x) >2$, from \eqref{eq:Gamma_Less_Than_Tau} and \eqref{eq:Tau_Less_Than_Gamma} we get
\begin{equation}
\gamma_R(x) - \epsilon \leq \tau_R(2x)  \leq \gamma_R(x) + \epsilon, \qquad \forall \epsilon < \min \{ \gamma_R(x)-2, \tau_R(2x) - 2  \}.
\end{equation}
Consequently, $\gamma_R(x) = \tau_R(2x)$ and the proof is complete.
\end{proof}

\subsection{A lemma about continued fractions}

\begin{lem}\label{lem:Continued_Fraction_Auxiliary}
Let $\rho \in \mathbb{R} \setminus \mathbb{Q}$ and its convergents by continued fractions $p_n/q_n$. Then, for any $n \in \mathbb{N}$, $q_n$ and $q_{n+1}$ are not both even. Consequently, there exists a subsequence of convergents $p_{n_j}/q_{n_j}$ such that $q_{n_j}$ is odd for all $j \in \mathbb{N}$.  
\end{lem}
\begin{proof}
By contradiction, let $N \in \mathbb{N} $ be such that $q_N$ and $q_{N+1}$ are both even. It is a basic fact of continued fractions \cite[Theorem 1]{Khinchin1964} that if the continued fraction of $\rho$ is $[a_0;a_1,a_2,\ldots]$, then the convergents satisfy $q_{n+1} = a_{n+1}q_n + q_{n-1}$ for every $n \geq 2$. In particular, 
\begin{equation}
q_{N-1} = q_{N+1} - a_{N+1}q_N = 0 \pmod{2},
\end{equation}
so $q_{N-1}$ is even. By induction, $q_n$ is even for every $n \leq N$. However, $p_0 / q_0 = a_0 = [\rho] \in \mathbb N$, so $q_0=1$, which is a contradiction. Hence, there are never two consecutive convergents with even denominator, and convergents with odd denominator are infinitely many.
\end{proof}

\section{The asymptotic behavior: heuristics}\label{sec:Asymptotics_Heuristics}

We now turn to the asymptotic behavior of Riemann's non-differentiable function $\phi$. Recall that we are looking for the precise behavior of $\phi(t_x + h) - \phi(t_x)$ when $h \to 0$, where $t_x = x/2\pi$.
%the definition of $t_x = x/2\pi$ is motivated by the $1/(2\pi)$-periodic property of $\phi$ \eqref{PeriodicityPhi}. We will be able to compute it when $x \in \mathbb Q$. 
We will always work with rationals $x=p/q$ such that $p$ and $q$ are coprime, and in that case we will often denote $t_{p/q}$ as $t_{p,q}$. 
In this section we explain the heuristics of this computation. The arguments here will be rigorously established in Sections~\ref{Section_BaseCases} and \ref{Section_Rationals}.

\subsection{Overview}\label{sec:Overview}

We mentioned in the introduction that Duistermaat \cite{Duistermaat1991} computed the asymptotic behavior of $\phi_D$ near rational points. For that, he first realized that the derivative of $\phi_D$ is directly related to the Jacobi $\theta$ function
\begin{equation}\label{JacobiTheta}
\theta(z) = \sum_{k \in \mathbb{Z}}{ e^{\pi i k^2 z} },  \qquad z \in \mathbb{H}  = \{ z \in \mathbb{C} \mid \operatorname{Im}(z) >0 \},
\end{equation}
because
\begin{equation}\label{eq:Phi_Duistermaat_Derivative}
\phi_D'(z) = \frac12 \left( \theta(z) - 1 \right), \qquad \forall z \in \mathbb{H}.
\end{equation}
The $\theta$ function interacts with the modular group $\Gamma$ of M\"obius transformations $\gamma$ that satisfy
\begin{equation}
\gamma(z) = \frac{az+b}{cz+d}, \qquad a,b,c,d \in \mathbb{Z}, \qquad ad-bc=1,
\end{equation}
which is a group under the operation of composition that is generated by the transformations
\begin{equation}
S(x) = 1/z \qquad \text{ and } \qquad T(z) = z + 1; \qquad \qquad \Gamma = \langle S, T \rangle.
\end{equation} 
It is well-known that the Jacobi $\theta$ function interacts very well with $S$, since the inversion identity
\begin{equation}\label{InversionOfTheta}
\theta \left( \frac{-1}{z} \right) = \sqrt{\frac{z}{i}}\, \theta(z), \qquad \forall z \in \mathbb H,
\end{equation}
holds with the principal branch of the square root. But $\theta$ interacts not with $T$ but with $T^2(z) = z+2$, since trivially
\begin{equation}\label{PeriodicityOfTheta}
\theta(z+2) = \theta(z), \qquad \forall z \in \mathbb H.
\end{equation}
Thus, the group linked to $\theta$ is the subgroup $\Gamma_\theta = \langle S,T^2 \rangle$, the so-called $\theta$-modular group. It can be equivalently written as
\begin{equation}\label{eq:Theta_Modular_Group}
\Gamma_\theta = \left\{ \,\,  \gamma(x) = \frac{ax+b}{cx+d} \quad \mid \quad a,b,c,d \in \mathbb{Z}, \, \, ad-bc = 1, \, \, a \equiv d \not\equiv b \equiv c \, (\text{mod } 2) \,\, \right\}.
\end{equation}
Properties~\eqref{InversionOfTheta} and \eqref{PeriodicityOfTheta} and the fact that $\Gamma_\theta$ is a group imply that for every $\gamma \in \Gamma_\theta$  there exists an identity relating $\theta(\gamma(z))$ with $\theta(z)$. In fact, it is
\begin{equation}\label{eq:Theta_Transformation_General}
\theta(\gamma(z)) = e_{\gamma}\,\sqrt{cz+d}\,\theta(z), \qquad \forall \gamma \in \Gamma_{\theta},
\end{equation}
where $e_\gamma$ is an eighth root of the unity depending only on $c$ and $d$. Details on the properties of the Jacobi $\theta$ function and of the modular group can be found in \cite{Apostol1990,SteinShakarchi2003}. 

Duistermaat used the transformation \eqref{eq:Theta_Transformation_General} in \eqref{eq:Phi_Duistermaat_Derivative} and integrated the identity to obtain an asymptotic expansion for $\phi_D(x) - \phi_D(r)$, where $r$ is the rational pole of the $\gamma \in \Gamma_\theta$ chosen. Here, as stated in the introduction, instead of using \eqref{FromPhiDuistermaatToPhi} to translate the asymptotic behavior for $\phi_D$ to $\phi$, we will compute the asymptotic behavior of $\phi$ directly.
%, for the reasons stated in the introduction, to have self-contained account of the whole program regarding the geometric analysis of $\phi$ and to
%This way, we will highlight interactions with some symmetries of the Schr\"odinger equation, the Gauss sums and the Talbot effect. 

In our case, the identity \eqref{eq:Phi_Duistermaat_Derivative} takes the form 
\begin{equation}\label{PhiWithTheta}
 \phi(t) = i \,\int_0^t{ \theta(-4\pi \tau)\,d\tau }, 
\end{equation}
at least formally because $\theta$ is not well-defined on $\mathbb{R}$. Then, the asymptotic at $t_x$ is
\begin{equation}\label{AsymptoticWithTheta}
\phi(t_x + h ) - \phi(t_x) = i \,\int_{t_x}^{t_x+h}{ \theta(-4\pi \tau)\,d\tau } =  i \,\int_{t_x}^{t_x+h}{ \psi(0, \tau)\,d\tau } ,
\end{equation} 
where $\psi$ is the Schr\"odinger solution \eqref{FreeSchrodingerSolution}. This expression, together with the $\theta$-modular transformations, will allow us to reduce the asymptotics around any rational to the behavior around either 0 or $t_{1,2}$. These two, on the other hand, can be computed by hand. 
%Thanks to the connection between the VFE and the NLS,
This reduction is related to 
%the pseudoconformal invariance of the Schr\"odinger solution \eqref{FreeSchrodingerSolution},
%the periodicity of the free Schr\"odinger equation \eqref{FreeSchrodingerEquation}.
 the Talbot effect and the generalized Gauss sums
\begin{equation}\label{GaussSum}
G(a,b,c) = \sum_{m = 0}^{c-1}{ e^{2\pi i \,\frac{a\,m^2 + b\,m}{c}} },  \qquad a,b \in \mathbb{Z}, \quad  c \in \mathbb{N}. 
\end{equation}  
Indeed, we are going to see in Subsection~\ref{Section_Heuristics} that the Talbot effect, which happens at the level of $\psi$, combined to the pseudoconformal invariance of the Schr\"odinger solution \eqref{FreeSchrodingerSolution} yields an iterative algorithm to reduce any Gauss sum $G(p,0,q)$ to the trivial $G(0,0,1)$ or $G(1,0,2)$. Thus, \eqref{AsymptoticWithTheta} suggests that this iterative algorithm can be translated to the level of $\phi$ to reduce the behavior around $t_{p,q}$ to either $t_{0,1}=0$ or $t_{1,2}$. 
In fact, these iterations will materialize in a single $\theta$-modular transformation, so the reduction will be the consequence of combining \eqref{eq:Theta_Transformation_General} and  \eqref{AsymptoticWithTheta}. However, the algorithm does not supply the transformation explicitly, so we will compute it in Subsection~\ref{Section_LookingForTransformations} following ideas of \cite{Jaffard1996}.

\subsection{Heuristics of the reduction: the Talbot effect and Gauss sums}\label{Section_Heuristics}

The Talbot effect is an optic phenomenon consisting in the interference caused by the diffracted light after crossing a grating with equidistant parallel slits. In 1836, Talbot \cite{Talbot1836} discovered a distance, called the Talbot distance nowadays, where the interference pattern matches the original grating. Later, it was discovered that in every fraction $p/q$ of the Talbot distance, the interference pattern is a grating with $q$ times as many slits as the original (see \cite{BerryMarzoliSchleich2001}). 

It turns out that the Talbot effect is mathematically expressed in terms of the solution $\psi$ \eqref{FreeSchrodingerSolution} to the Schr\"odinger equation \cite{BerryKlein1996,MatsutaniOnishi2003}. More precisely, 
\begin{equation}\label{TalbotEffect}
\psi(s,t_{p,q} ) =  \sum_{k \in \mathbb{Z}}{ e^{2\pi i \left( k s - k^2\frac{p}{q}  \right)} } = \frac{1}{q} \, \sum_{k \in \mathbb{Z}}{ \sum_{r=0}^{q-1}{ G(-p,r,q)\, \delta\left(s-k-\frac{r}{q}\right) } }, 
\end{equation} 
where $G(-p,r,q)$ are Gauss sums \eqref{GaussSum}, see \cite[Section 3.3]{delaHozVega2014} for the details.
%It was shown in \cite{delaHozVega2014} that this phenomenon appears in the evolution of polygonal vortex filaments, 

% and the Talbot carpet. Indeed, in the three-dimensional space, let the grating be in the plane $OXY$ so that the slits are parallel to the axis $Y$ and light travels in direction $Z$. Then, if we project in $Y=0$, the space variable $s$ and time variable $t$ correspond to $X$ and $Z$ respectively, and the initial condition in \eqref{FreeSchrodingerEquation} represents the grating, where each Dirac delta stands for a slit. According to  \eqref{TalbotEffect}, at time $t_{p,q}$ equally separated $q$ times more deltas have formed, which corresponds to $q$ times more slits having formed at distance $p/q$ in the experiment. An exception occurs when $q \equiv 2 (\text{mod } 4)$, when half of the Gauss sums vanish and there are only $q/2$ times more deltas. The case $q=1$ represents the Talbot distance, or the period, where the grating is reproduced exactly, while when $q=2$ the number of slits is the same as in the beginning, but positions are switched so that a slit is formed in the middle of every two original slits. This last event is known in the literature as the axis switching phenomenon (see, for instance, \cite{GutmarkGrinstein1999}).

The Talbot effect \eqref{TalbotEffect} and the prseudoconformal symmetry of the Schr\"odinger equation can be used to compute Gauss sums iteratively. The basic idea is that a symmetry together with an invariant initial datum yields an invariance for the corresponding solution, in case uniqueness of solutions is granted. 
For example, the free Schr\"odinger equation is translation invariant: if $u(s,t)$ is a solution, then so is $u(s+1,t)$. This symmetry takes the initial condition $u(s,0)$ to $u(s+1,0)$. In \eqref{FreeSchrodingerEquation}, $\psi_0(s) = \psi_0(s+1)$, so assuming uniqueness, the two solutions must also coincide, so $\psi(s,t) = \psi(s+1,t)$.

We repeat this procedure with the pseudoconformal symmetry
\begin{equation}\label{PseudoconformalTransformation}
\mathcal{P}u(s,t) = \frac{1}{\sqrt{4\pi i t}} \, \overline{u}\left( \frac{s}{t},\frac{1}{t} \right) \,e^{is^2/(4t)}, \qquad \mathcal{P}u(s,0) = \mathcal{F}^{-1}\left( \overline{u}(4\pi \cdot,0) \right)(s) = \frac{1}{4\pi}\mathcal{F}^{-1}\overline{u}\left( \frac{s}{4\pi},0 \right),
\end{equation}
where the bar represents complex conjugation.
%and $\sqrt{i} = ( 1+\operatorname{sign}(t)\,i) / \sqrt{2}$ is determined by the fundamental solution of the Schr\"odinger equation. 
Due to the Poisson summation formula, the initial datum $\psi_0(s) =\psi(s,0)$ satisfies $\widehat{\psi_0} = \psi_0 = \overline{\psi_0}$, so 
\begin{equation}
\mathcal{P}\psi_0(s) = \frac{1}{4\pi}\psi_0\left( \frac{s}{4\pi}\right) . 
\end{equation}  
Then, if uniqueness of solution is assumed, we get 
\begin{equation}
 \mathcal{P}\psi(s,t) =   \frac{1}{4\pi}\psi\left( \frac{s}{4\pi},\frac{t}{(4\pi)^2} \right).  
\end{equation}
Rearranging the above leads to the pseudoconformal invariance of $\psi$,
\begin{equation}\label{Pseudoconformal_Invariance}
\psi(s,t) = \frac{1}{(4\pi i t)^{1/2}}\, e^{i s^2/(4t)}\, \overline{\psi}\left( \frac{s}{4\pi t}, \frac{1}{(4\pi)^2t} \right).
\end{equation}
%The determination of the Gauss sums is a classic problem of number theory first tackled by Gauss in the beginning of the 19th century. Here, using symmetries of the Schr\"odinger equation, we are able to obtain equalities between two \eqref{TalbotEffect}-like expressions and compare the coefficients of the Dirac deltas so that we obtain an equality between two different Gauss sums. This procedure is not only an alternative way to compute Gauss sums by an iterative reduction process, but also hints the possibility of translating the behaviour of $\phi$ at any $t_{p,q}$ to either 0 or $t_{1,2}$.
The key point is that \eqref{Pseudoconformal_Invariance} allows the reduction
\begin{equation}\label{Pseudoconformal_Reduction}
t_{p,q} \to \frac{1}{(4\pi)^2}\,\frac{1}{t_{p,q}} = \frac{1}{2\pi}\,\frac{q}{4p} = t_{q,4p}.
\end{equation} 
To see the effect of this at the level of Gauss sums, evaluate \eqref{Pseudoconformal_Invariance} in $t_{p,q}$ and use the Talbot effect \eqref{TalbotEffect} to get
\begin{equation}\label{eq:TalbotAndPseudoconformal}
 \frac{1}{q}\,\sum_{k \in \mathbb{Z}} \sum_{r=0}^{q-1} G(-p,r,q)\,\delta\left(s - k -\frac{r}{q}\right)   =    \frac{e^{\frac{i\pi}{2}\frac{q}{p}s^2}}{2\,\sqrt{2ipq}} \,\sum_{k\in\mathbb{Z}}\sum_{r=0}^{4p-1} \overline{G(-q,r,4p)}\,\delta\left( s-\frac{2p}{q}k - \frac{r}{2q} \right).
\end{equation}
Compare the coefficients of the respective Dirac deltas at $s=0$ to get the well-known reciprocity formula for Gauss sums,
\begin{equation}\label{eq:Reciprocity_Formula}
G(p,0,q) = \sqrt{\frac{q}{p}}\, \frac{1+i}{4}\,G(-q,0,4p),
\end{equation} 
which can be found, for instance, in \cite[Theorem 1.2.2]{BerndtEvansWilliams1998}.  

Gauss sums are easy to compute by hand when $q$ is small. For instance, \eqref{eq:Reciprocity_Formula} immediately implies the non-trivial $G(1,0,q) =  \sqrt{q} \,(1+i)(1+(-i)^q)/2$ for every $q \in \mathbb{N}$. In the same way,  we may combine it with the trivial modular property  
\begin{equation}\label{eq:Modularity_Formula}
 G(a,0,c) = G(a(\text{mod } c),0,c),
\end{equation}
to compute $G(p,0,q)$ iteratively. We do that in Algorithm~\ref{thm:Algorithm_For_Reduction}. We do not take care of the multiplying factors coming from each time we use the reciprocity formula \eqref{eq:Reciprocity_Formula}, but just control the reduction of the variables $(p,q)$ of the Gauss sums. 

\begin{alg}\label{thm:Algorithm_For_Reduction}
Let $p,q \in \mathbb{N}$ coprime integers such that $q \neq 1,2,4$ and $p<q$. Denote by $R$ the reciprocity formula \eqref{eq:Reciprocity_Formula} and by $M$ the modularity formula \eqref{eq:Modularity_Formula}. 
\begin{itemize}
	\item If $p < q/2$,  do $(p,q) \xrightarrow{R} (-q,4p) \xrightarrow{M} (4p-q,4p)$. 
	\begin{itemize}
		\item If $p < q/4$, then $4p < q$. The denominator has been reduced.
		\item If $q/4 < p < q/2$, iterate again $(4p-q,4p) \xrightarrow{R} (-p,4p-q) \xrightarrow{M} (3p-q,4p-q)$. And $0 <4p-q< q$. The denominator has been reduced.
	\end{itemize}
	\item If $q/2 < p < q$, do $(p,q) \xrightarrow{M} (p-q,q) \xrightarrow{R} (q,4(q-p))$.
	\begin{itemize}
		\item If $p > 3q/4$, then $4(q-p) < q$. The denominator has been reduced.
		\item If $q/2  < p < 3q/4$, iterate again $(q,4(q-p)) \xrightarrow{M} (4p - 3q,4(q-p)) \xrightarrow{R} (q-p,3q - 4p)$, where $3q-4p < q$. The denominator has been reduced.
	\end{itemize}	 
\end{itemize} 
If $q=4$, then $(p,4) \xrightarrow{R} (-4,4p) = (-1,p) \xrightarrow{M} (p-1,p)$, where $p=1$ or $p=3$. Therefore, the denominator $q$ can always be reduced to $q=1$ or $q=2$. When $q=2$, then $(1,2) \xrightarrow{R} (-2,4) = (-1,2) \xrightarrow{M} (1,2)$, so the algorithm takes $q=2$ to itself. 
\end{alg}

\begin{rem}
In the same way that the reciprocity formula \eqref{eq:Reciprocity_Formula} is a consequence of the pseudoconformal invariance \eqref{Pseudoconformal_Invariance}, the modular property \eqref{eq:Modularity_Formula} can be seen a consequence of the time periodicity of $\psi$,
\begin{equation}\label{Periodic_Invariance}
\psi(s,t) = \psi(s,t+1/2\pi),
\end{equation}
and corresponds to the time transformation
\begin{equation}\label{eq:Modular_Reduction}
t_{p,q} \to t_{p,q} + \frac{k}{2\pi} = \frac{1}{2\pi}\,\left( \frac{p}{q} + k \right) = t_{p+kq,q}, \qquad \forall k \in \mathbb{Z}.
\end{equation} 
\end{rem}

In short, Algorithm~\ref{thm:Algorithm_For_Reduction} shows that for every irreducible rational number $p/q$ there exists a transformation $\gamma$, formed by several combinations of \eqref{Pseudoconformal_Reduction} and \eqref{eq:Modular_Reduction}, and which has attached two other transformations $a_\gamma$ and $b_\gamma$ coming from the corresponding \eqref{Pseudoconformal_Invariance} and \eqref{Periodic_Invariance}, such that
\begin{equation}\label{eq:Identity_After_Transformation}
\psi(s,t) = a_\gamma(s,t)\,\psi\left( b_\gamma(s,t),\gamma(t) \right)
\end{equation}
and either $\gamma(t_{p,q}) = t_{0,1} = 0$ or $\gamma(t_{p,q})=t_{1,2}$. This identity can now be plugged in \eqref{AsymptoticWithTheta}, so a change of variables $\gamma(t)=\tau$ should lead to the asymptotic behavior around $0$ or $t_{1,2}$. 

At this stage, we do not know an explicit expression for $\gamma$,
% so we still cannot do these computations. 
but we can guess the nature of $\gamma$ anyways. For that, rewrite \eqref{AsymptoticWithTheta} by changing variables $r = 2\pi \tau$ as
\begin{equation}\label{eq:Asymptotic_Ready_For_Reduction}
\phi(t_x+ h) - \phi(t_x) = i\, \int_{x}^{x+2\pi h}{  \psi(0,r/2\pi)\,dr} = \frac{i}{2\pi}\,\int_x^{x+2\pi h}{ \theta(-2r)\,dr }.
\end{equation} 
This way, it is adapted to the setting of Algorithm~\ref{thm:Algorithm_For_Reduction} with $r \in (x,x+2\pi h)$ in the same scale as $p/q$. That means that the time transformations coming from \eqref{Pseudoconformal_Invariance} and \eqref{Periodic_Invariance} are applied to $\eta(r) = \theta(-2r)$. According to \eqref{Pseudoconformal_Reduction}, reciprocity changes $\eta(r) \to \eta(-1/4r)$, that is, $\theta(r) \to \theta(-1/r)$. On the other hand, in view of \eqref{eq:Modular_Reduction} with $k=1$, modularity changes $\eta(r) \to \eta(r+1)$, that is, $\theta(r) \to \theta(r+2)$. 
%Eventually, several combinations of these transformations build $\tilde{\mu}(x) = \mu(x/(2\pi))$.
These two transformations,
\begin{equation}\label{eq:Transformations_Rescaled}
r \to 1/r \qquad \text{ and } \qquad  r \to r + 2,
\end{equation}
are precisely the generators of the $\theta$-modular group $\Gamma_\theta$ \eqref{eq:Theta_Modular_Group}. Since $\gamma$ is a combination of both, then it must be a $\theta$-modular transformation $\gamma \in \Gamma_\theta$. Observe that we have changed the scale in \eqref{eq:Asymptotic_Ready_For_Reduction} again, with a change of variables $2r=\sigma$. The proper setting is now 
\begin{equation}\label{eq:Asymptotic_Ready_For_Reduction_With_Modular_Group}
\phi(t_x+ h) - \phi(t_x) = \frac{i}{4\pi}\,\int_{2x}^{2x+4\pi h}{ \theta(- \sigma)\,d\sigma },
\end{equation}
and for $x = p/q$, since the reduction will yield asymptotics at $0$ or $t_{1,2}$, then either $\gamma(2p/q)=0$ or $\gamma(2p/q)=1$ will hold.
%and $\gamma(x) = \tilde{\mu}(x/2) = \mu(x/(4\pi))$, such that $\gamma\in\Gamma_\theta$ and for $\tilde{p}/\tilde{q}$ we have either $\gamma(\tilde{p}/\tilde{q}) = 0$ or $\gamma(\tilde{p}/\tilde{q}) = 1$.
From now on, we will denote by $\tilde{p}/\tilde{q}$ the irreducible fraction of $2p/q$, so that
\begin{equation}\label{eq:Definition_Of_TildePTildeQ}
\begin{array}{lll}
 \tilde{p} = 2p, & \quad \tilde{q} = q, & \qquad \text{ if } q \text{ is odd, } \\
 \tilde{p} = p, & \quad \tilde{q} = q/2, & \qquad \text{ if } q \text{ is even. }
\end{array}
\end{equation}

At this point, we can guess which rational numbers can be sent to 0 and which cannot. Assume both $\tilde{p},\tilde{q}$ are odd and that $\gamma \in \Gamma_{\theta}$ is such that $\gamma(\tilde{p}/\tilde{q}) = 0$. The coefficients in the numerator of $\gamma$, $a$ and $b$  (see \eqref{eq:Theta_Modular_Group}), are coprime, so either $a = \tilde{q}$ and $b = -\tilde{p}$ or $a = -\tilde{q}$ and $b = \tilde{p}$ must hold. But then the parity condition in \eqref{eq:Theta_Modular_Group} is not kept, hence $\gamma$ does not exist. These points are precisely corresponding to $p/q$ with $q \equiv 2\, (\text{mod } 4)$, because then $p$ is odd and $\tilde{p}/\tilde{q} = p/(q/2)$, where $q/2$ is odd. On the other hand, if $q \equiv 0 \, (\text{mod } 4)$, then $\tilde{p}/\tilde{q} = p/(q/2)$ with $p$ odd and $q/2$ even, and if $q \equiv 1,3 \, (\text{mod } 4)$, then $\tilde{p}/\tilde{q} = 2p/q$ with $2p$ even and $q$ odd. 

In Subsection~\ref{Section_LookingForTransformations}, we prove that the general scheme for the $\theta$-modular transformations corresponding to $t_{p,q}$ is 
\begin{equation}\label{eq:Classification_Of_Rationals}
\begin{array}{lll}
q \text{ odd } & \Longrightarrow & \tilde{p} = 2p, \quad \tilde{q} = q, \quad \,\,\,  \exists \gamma \in \Gamma_{\theta} \text{ such that } \gamma(\tilde{p}/\tilde{q}) = 0.  \\
q \equiv 0 \, (\text{mod } 4) & \Longrightarrow & \tilde{p} = p, \quad \tilde{q} = q/2, \quad \exists \gamma \in \Gamma_{\theta} \text{ such that } \gamma(\tilde{p}/\tilde{q}) = 0. \\
q \equiv 2 \, (\text{mod } 4) & \Longrightarrow & \tilde{p} = p, \quad \tilde{q} = q/2, \quad \exists \gamma \in \Gamma_{\theta} \text{ such that } \gamma(\tilde{p}/\tilde{q}) = 1.
\end{array}
\end{equation} 
We will also compute these transformations.

\subsection{Formal reduction and $\theta$-modular functions}\label{Section_LookingForTransformations}

We now compute the $\theta$-modular transformations of classification \eqref{eq:Classification_Of_Rationals} explicitly, which were essentially given in \cite{Jaffard1996}. Then, combining them with \eqref{eq:Asymptotic_Ready_For_Reduction_With_Modular_Group}, we will  reduce the asymptotics around $t_{p,q}$ to either 0 or $t_{1,2}$ formally. The conclusions, though heuristic, are very enlightening.

We determine the coefficients $a,b,c,d$ of $\gamma \in \Gamma_{\theta}$ as in \eqref{eq:Theta_Modular_Group} using continued fractions. Let $\tilde{p}_n/\tilde{q}_n$ be the $n$-th convergent of $\tilde{p}/\tilde{q}$ by continued fractions. As a rational number, it has finitely many convergents, so there exists $N \in \mathbb{N}$ such that $\tilde{p}/\tilde{q} = \tilde{p}_N/\tilde{q}_N$.  Also, recall that $\tilde{p}_n\,\tilde{q}_{n-1} - \tilde{q}_n\,\tilde{p}_{n-1} = (-1)^{n-1}$ for every $n \leq N$. Details about continued fractions can be found in \cite{Khinchin1964}.

%The continued fraction of the a rational $\tilde{p}/\tilde{q}$ is finite, so there exists $N \in \mathbb{N}$ such that $\tilde{p}/\tilde{q} = [a_0;a_1,\ldots,a_N]$, where $a_n \in \mathbb{N}$ for all $n \in \mathbb{N}$. For any $n \leq N$, the $n$-th convergent is a rational $\tilde{p}_n/\tilde{q}_n = [a_0;a_1,\ldots,a_n] $ satisfying $|\tilde{p}/\tilde{q} - \tilde{p}_n/\tilde{q}_n| < \tilde{q}_n^{-2}$. Also, $\tilde{p}_n\,\tilde{q}_{n-1} - \tilde{q}_n\,\tilde{p}_{n-1} = (-1)^{n-1}$ for every $n \leq N$. The reader may consult \cite{Khinchin1964} for further details.

\subsubsection{Transformation for rationals $p/q$ such that $q \equiv 0,1,3 \pmod{4}$}
% $\tilde{p}$ and $\tilde{q}$ are not both odd
\label{Subsection_TransformationNotBothOdd}

According to \eqref{eq:Classification_Of_Rationals}, these rationals can be sent to 0. Indeed, $\tilde{p}$ and $\tilde{q}$ are not both odd, so choose
\begin{equation}
a = \tilde{q}, \qquad b = -\tilde{p}.
\end{equation} 
%except maybe the sign. 
Since $\tilde{p} = \tilde{p}_N$ and $\tilde{q}=\tilde{q}_N$, the other coefficients will depend on $\tilde{p}_{N-1}$ and $\tilde{q}_{N-1}$:
\begin{itemize}
	\item If $\tilde{p}_{N-1}$ and $\tilde{q}_{N-1}$ are not both odd, we choose 
	\[ c = (-1)^{N-1}\,\tilde{q}_{N-1}, \qquad d = (-1)^N\,\tilde{p}_{N-1}, \]
	so that $ad-bc = (-1)^N \left( \tilde{q}\,\tilde{p}_{N-1} - \tilde{p}\,\tilde{q}_{N-1} \right) = (-1)^{2N} = 1$.
	\item If $\tilde{p}_{N-1}$ and $\tilde{q}_{N-1}$ are both odd, the above does not satisfy the parity conditions, so choose
	\[ c = (-1)^{N-1}\,\tilde{q}_{N-1} + \tilde{q}, \qquad d = (-1)^N\,\tilde{p}_{N-1}-\tilde{p}. \]
\end{itemize}

\begin{rem}\label{RemarkForC013}
The choice of $c$ and $d$ is not unique. Indeed, parity and the determinant are preserved with $c' = c + 2k\tilde{q}$ and $d' = d - 2k\tilde{p}$ for any $k \in \mathbb{Z}$. If $k=1$, we may work with $\tilde{q} < c < 4\tilde{q}$ in both cases. If $k=-1$ in the first case and $k=-2$ in the second one, we may also work with $-4\tilde{q} < c < -\tilde{q}$. 
\end{rem}

\subsubsection{Transformation for rationals $p/q$ such that $q \equiv 2 \pmod{4}$}
%$\boldsymbol{\tilde{p}}$ and $\boldsymbol{\tilde{q}}$ are both odd
\label{Subsection_TransformationBothOdd}
According to \eqref{eq:Classification_Of_Rationals}, they cannot be sent to 0. In this case, both $\tilde{p}$ and $\tilde{q}$ are odd, so choose
\[ a = (-1)^{N-1}\,\tilde{q}_{N-1} + \tilde{q}, \quad b = (-1)^N\tilde{p}_{N-1}-\tilde{p}, \quad c = (-1)^{N-1}\tilde{q}_{N-1}, \quad d= (-1)^N\,\tilde{p}_{N-1}.  \]
Indeed, $\tilde{p}_{N-1}$ and $\tilde{q}_{N-1}$ cannot both be odd, so parity conditions are preserved. Also $ad-bc = 1$. One can easily check that $\gamma(\tilde{p}/\tilde{q}) =1$. 

\begin{rem}\label{RemarkForC2}
Here too, the choice of $a,b,c,d$ is not unique, since all properties are preserved if 
\[ \begin{array}{ll}
a = (-1)^{N-1}\,\tilde{q}_{N-1} + (2k+1)\tilde{q}, & b = (-1)^N\tilde{p}_{N-1}-(2k+1)\tilde{p}, \\
c = (-1)^{N-1}\tilde{q}_{N-1} + 2k\tilde{q}, & d= (-1)^N\,p_{N-1}-2k\tilde{p},
\end{array}  \] 
for any $k \in \mathbb{Z}$. With $k=1$, we may assume $\tilde{q} < c < 3\tilde{q}$, and with $k=-1$, we may work with $-3\tilde{q} < c < -\tilde{q}$.
\end{rem}

\subsubsection{Formal reduction}\label{Subsection_Reduction}

Once we have the transformations, let us use them in \eqref{eq:Asymptotic_Ready_For_Reduction_With_Modular_Group} to reduce from $t_{p,q}$ to either 0 or $t_{1,2}$ formally. 

% This formalism comes from \eqref{AsymptoticWithTheta}, which as remarked previously, is not correct since $\theta$ is not defined on the real axis. However, the conclusions of the following lines can be properly justified because one can work with the limits
%\begin{equation*}%\label{PhiWithLimit}
%\phi\left( t \right) = \frac{i}{4\pi}\,\lim_{\epsilon \to 0^+} \int_0^{4\pi t}{\theta(-\tau + i\epsilon)\,d\tau},
%\end{equation*}
%and therefore, with
%\begin{equation}\label{PhiWithLimit}
% \phi\left( t_{p,q} + h \right) - \phi(t_{p,q}) = \frac{i}{4\pi}\,\lim_{\epsilon \to 0^+} \int_{\tilde{p}/\tilde{q}}^{\tilde{p}/\tilde{q} + 4\pi h}{\theta(-\tau + i\epsilon)\,d\tau}.
%\end{equation}
%%for example working with $\theta(x+i\epsilon)$ in the upper half-plane with $x\in\mathbb{R}$ and $\epsilon >0$ and taking the limit $\epsilon\to 0$.
%%An alternative is to work with $\theta$ in the sense of tempered distributions. 
%We leave the rigorous proof for the upcoming sections and we now focus to explain the intuition of the reduction itself. 

We begin with $0 < p \leq q$ coprime such that $q \equiv 0,1, 3 \pmod{4}$. We just saw that there exists $\gamma \in \Gamma_{\theta}$ such that $\gamma(\tilde{p}/\tilde{q}) = 0$.  According to \eqref{eq:Asymptotic_Ready_For_Reduction_With_Modular_Group}, for $h \in \mathbb{R}$ we have
\begin{equation}\label{eq:Asymptotic_Ready_For_Reduction_With_Modular_Group_Rationals}
\phi(t_{p,q}+h) - \phi(t_{p,q}) 
% = i \,\int_{t_{p,q}}^{t_{p,q}+h}{ \theta(-4\pi \tau)\,d\tau } 
= \frac{i}{4\pi} \,\int_{\tilde{p}/\tilde{q}}^{\tilde{p}/\tilde{q} + 4\pi h}{ \theta(-\sigma)\,d\sigma }. 
\end{equation} 
Conjugate and use the transformation \eqref{eq:Theta_Transformation_General} with the $\gamma$ above so that
\begin{equation}\label{eq:Before_Changing_Variables}
\overline{ \phi(t_{p,q}+h) - \phi(t_{p,q}) } = \frac{\overline{i\,e_{\gamma}}}{4\pi}\, \int_{\tilde{p}/\tilde{q}}^{\tilde{p}/\tilde{q} + 4\pi h}{ \frac{\theta(\gamma(\sigma))}{ \sqrt{c\sigma+d} } \,d\sigma }. 
\end{equation}  
Now, change variables $\gamma(\sigma) = r$. Since  $a = \tilde{q},\, b=-\tilde{p}$ and $ad-bc = 1$, we have
\begin{equation}\label{eq:Changing_Variables}
  \gamma(x) = \frac{ax+b}{cx+d} \qquad \Longrightarrow \qquad \gamma^{-1}(x) = \frac{dx-b}{-cx+a}, \quad \gamma'(x) = \frac{1}{(cx+d)^2}. 
\end{equation}
Then, the boundaries of the integral become $\gamma(\tilde{p}/\tilde{q})=0$ and 
\begin{equation}\label{eq:Upper_Boundary_Of_Integral}
\gamma(\tilde{p}/\tilde{q} + 4\pi h) = \frac{4\pi \tilde{q}^2h}{1+4\pi c \tilde{q} h}.
\end{equation}
At this point, the cases $h>0$ and $h<0$ have to be considered separately. To avoid a null denominator, if $h\geq0$, following Subsections~\ref{Subsection_TransformationNotBothOdd} and \ref{Subsection_TransformationBothOdd} we let $c = c_+$ be such that $\tilde{q} < c_+ < 4\tilde{q}$. On the other hand, if $h<0$, choose $c = c_-$ such that $-4\tilde{q} < c_- < -\tilde{q}$. This way, we have  $4\pi c \tilde{q} h \geq 0$ in both cases. With \eqref{eq:Upper_Boundary_Of_Integral} in mind, define
\begin{equation}\label{eq:Definition_Of_B}
 b(h) = \frac{\tilde{q}^2h}{1+4\pi c_\pm \tilde{q} h} = \left\{
\begin{array}{ll}
\frac{\tilde{q}^2h}{1+4\pi c_{+} \tilde{q} h}, & \text{when } h \geq 0, \\
\frac{\tilde{q}^2h}{1+4\pi c_{-} \tilde{q} h}, & \text{when } h < 0.
\end{array}
\right.
\end{equation}  
%If we write $h = -|h|$ if $h<0$, 
Then, \eqref{eq:Before_Changing_Variables} turns into
\begin{equation}\label{eq:After_Changing_Variables_013}
\overline{ \phi(t_{p,q} + h ) - \phi(t_{p,q}) } = \frac{\overline{i\,e_{\gamma}}}{4\pi}\, \int_0^{4\pi b(h)}{  \frac{\theta(r)}{(\tilde{q}-c_{\pm}r)^{3/2}}\, dr    }, \qquad \text{ for all } h. 
\end{equation} 
When $|h|$ is small, $b(h)$ behaves like $\tilde{q}^2h$, so the variable $r$ of the integral is small and $\tilde{q} - cr$ is similar to $\tilde{q}$. Thus, by \eqref{eq:Asymptotic_Ready_For_Reduction_With_Modular_Group_Rationals}, the asymptotic around $t_{p,q}$ will behave approximately as
\begin{equation}\label{eq:Approximated_Asymptotic_013}
\phi(t_{p,q} + h ) - \phi(t_{p,q}) \approx  \frac{e_{\gamma}}{\tilde{q}^{3/2}}\, \frac{i}{4\pi} \,  \int_0^{ 4\pi \tilde{q}^2 h }{  \theta(-r)\, dr    }  =  \frac{e_{\gamma}}{\tilde{q}^{3/2}}\,\phi( \tilde{q}^2\,h). %,
\end{equation} 
This means that when $h \to 0$, the behavior of $\phi$ around $t_{p,q}$ is essentially the same as around 0, except that we need to rescale by $\tilde{q}^2$ in the variable and by $\tilde{q}^{-3/2}$ in the image.

On the other hand, if   $q \equiv 2 \pmod{4}$, there exists $\gamma \in \Gamma_{\theta}$ such that $\gamma(\tilde{p}/\tilde{q}) = 1$. The same steps as before lead to 
\begin{equation}\label{eq:After_Changing_Variables_2}
\overline{ \phi(t_{p,q} +  h) - \phi(t_{p,q}) } = \frac{\overline{i\,e_{\gamma}}}{4\pi}\, \int_1^{1 + 4\pi b(h)}{  \frac{\theta(r)}{(\tilde{q}-c_{\pm}(r-1))^{3/2}}\, dr  } .
\end{equation} 
Like before, when $|h|$ is small we have $b(h) \approx \tilde{q}^2h$, so
\begin{equation}\label{eq:Approximated_Asymptotic_2}
\phi(t_{p,q}+h) - \phi(t_{p,q})  \approx  \frac{\,e_{\gamma}}{\tilde{q}^{3/2}} \, \frac{i}{4\pi}\, \int_1^{1 + 4\pi \tilde{q}^2 h}{ \theta(-r)\, dr  }  =  \frac{e_{\gamma}}{\tilde{q}^{3/2}}\,\left(  \phi(t_{1,2} + \tilde{q}^2\,h) - \phi(t_{1,2}) \right).
\end{equation}
Thus, up to the same scaling as before, the behavior of $\phi$ around $t_{p,q}$ is essentially the same as around $t_{1,2}$ when $h \to 0$.

%\begin{rem}
%The formal expression \eqref{eq:Asymptotic_Ready_For_Reduction_With_Modular_Group_Rationals} still needs to be given a rigorous sense. That can be done if we work with the limit
%%\begin{equation}%\label{PhiWithLimit}
%%\phi\left( t \right) = \frac{i}{4\pi}\,\lim_{\epsilon \to 0^+} \int_0^{4\pi t}{\theta(-\tau + i\epsilon)\,d\tau},
%%\end{equation}
%%and therefore, with
%\begin{equation}\label{eq:Asymptotic_With_Limit}
% \phi\left( t_{p,q} + h \right) - \phi(t_{p,q}) = \frac{i}{4\pi}\,\lim_{\epsilon \to 0^+} \int_{\tilde{p}/\tilde{q}}^{\tilde{p}/\tilde{q} + 4\pi h}{\theta(-\tau + i\epsilon)\,d\tau}.
%\end{equation}
%We will make this rigorous in Section~\ref{Section_Rationals}.
%\end{rem}
We will make make this formal reduction rigorous in Section~\ref{Section_Rationals}. However, that will be of no use if we do not know how $\phi$ behaves around 0 and $t_{1,2}$. We devote Section~\ref{Section_BaseCases} to compute the asymptotic behavior of $\phi$ around those two points by hand.

\section{Asymptotic behavior around 0 and $t_{1,2}$}\label{Section_BaseCases}

\subsection{Asymptotic behavior around $0$}\label{Section_AsymptoticIn0}

Since $\phi(0)=0$, we need to compute an asymptotic expression for $\phi(h)$. The main idea, which can be traced back to Smith \cite{Smith1972}, is to use the Poisson summation formula.
We begin assuming  $h > 0$ and writing
\begin{equation}\label{eq:Definition_Of_G}
\phi(h) = -h\,\sum_{k \in \mathbb{Z}}{ g(2\pi k \sqrt{h}) }, \qquad \text{ where } \qquad  g(x) = \frac{e^{-ix^2}-1}{x^2}.
\end{equation}  
The Poisson summation formula (see \cite[Theorem 3.1.17]{Grafakos2009}) gives
\begin{equation}\label{PoissonSummationFormulaUsed}
\phi(h) = -\frac{\sqrt{h}}{2\pi}\,\sum_{k \in \mathbb{Z}}{\widehat{g} \left( \frac{k}{2\pi \sqrt{h}} \right)}
\end{equation} 
if $|g(x)| + |\widehat{g}(x)| \leq C(1+|x|)^{-1-\delta}$ for some $C,\delta > 0$. The function $g$ satisfies that property because it is analytic, so bounded in any compact set, and it decreases as $|x|^{-2}$ when $|x| \to \infty$. To prove that property for $\widehat{g}$, we need the following lemma, very similar to \cite[Lemma 1]{OskolkovChakhkiev2010}.

\begin{lem}\label{thm:Fourier_Transform_Lemma}
The Fourier transform of $g$ defined in \eqref{eq:Definition_Of_G} is 
\begin{equation}
\widehat{g}(\xi) = 2\pi^2\,|\xi|\, \operatorname{erfc}\left( \frac{1-i}{\sqrt{2}}\,\pi\,|\xi| \right) - \sqrt{2\pi}\,(1+i)\,e^{i\pi^2 \xi^2},\qquad \forall \xi \in \mathbb{R}, 
\end{equation} 
where $\operatorname{erfc}(z) = 1-\operatorname{erf}(z)$ stands for the complementary error function and  $\operatorname{erf}(z) = \frac{2}{\sqrt{\pi}}\int_0^z{ e^{-w^2}\,dw }$ is the error function for $z \in \mathbb C$. Its asymptotic expansion for $x \in \mathbb{R}$ at infinity is
\begin{equation}\label{eq:Error_Function_Asymptotic_Expansion}
\operatorname{erfc}(x) = \frac{e^{-x^2}}{\sqrt{\pi}}\, \left( \frac{1}{x} +  \sum_{n=1}^N{ (-1)^n\, \frac{(2n-1)!!}{2^n\,x^{2n+1}}} \right) + O\left( \frac{1}{x^{2N+3}} \right), \qquad \forall N \in \mathbb{N}.
\end{equation}  
\end{lem}
\begin{rem}
The integral of the holomorphic function $e^{-w^2}$, $w \in \mathbb{C}$ in the definition of the error function can be computed along any path connecting 0 and $z$. 
\end{rem}
\begin{proof}
From the definition of $g$, integrating by parts we get
\begin{equation}
\widehat{g}(\xi) = -2i \, \int_{\mathbb{R}}{ e^{-ix^2}e^{-2\pi i \xi x}\,dx  } + 2 \pi i \xi \, \int_{\mathbb{R}}{  \frac{e^{-2\pi i \xi x}}{x}\,dx  } - 2\pi i \xi \int_{\mathbb{R}}{ \frac{e^{-ix^2}}{x}e^{-2\pi i \xi x}\,dx }. 
\end{equation} 
The first two integrals are the well-known
\begin{equation}\label{eq:Fourier_Transform_Of_Gaussian_And_1/x}
\mathcal{F}_x\left( e^{-ix^2} \right)(\xi) = \sqrt{\pi}\,\frac{1-i}{\sqrt{2}} \, e^{i\pi^2\xi^2},  \qquad    \mathcal{F}_x\left(1/x \right)(\xi) = -\pi i \operatorname{sign}(\xi),
\end{equation}
while the third one is the convolution of both of them, that is, 
%except the constants outside the integral, is
\begin{equation}
\int_\mathbb{R} e^{i\pi^2x^2}\,\operatorname{sign}(\xi-x)\,dx =  \int_{-\infty}^{\xi} e^{i\pi^2x^2}\,dx - \int_\xi^\infty e^{i\pi^2x^2}\,dx  = \operatorname{sign}(\xi)\,\int_{-|\xi|}^{|\xi|}e^{i\pi^2x^2}\,dx.
\end{equation}
 Hence, 
\begin{equation}
\widehat{g}(\xi) = -\sqrt{2\pi}\,(1+i)\,e^{i\pi^2\xi^2}  + 2\pi^2 |\xi| - 4\pi^2\sqrt{\pi}\,\frac{1-i}{\sqrt{2}}|\xi|\,\int_0^{|\xi|}{e^{i\pi^2y^2}\,dy}.
\end{equation}     
The last integral is essentially $\operatorname{erf}(\textstyle{\frac{1-i}{\sqrt{2}}}\,\pi |\xi|)$, because with the path $\eta(t) = \textstyle{\frac{1-i}{\sqrt{2}}}\,\pi t$, $t \in (0,|\xi|)$ we get 
\begin{equation}
\operatorname{erf}\left(\frac{1-i}{\sqrt{2}}\,\pi |\xi|\right) = \frac{2}{\sqrt{\pi}}\, \int_0^{|\xi|} e^{-\eta(t)^2}\,\eta'(t)\,dt = 2\sqrt{\pi}\,\frac{1-i}{\sqrt{2}} \, \int_0^{|\xi|} e^{i \pi^2 t^2}\,dt.
\end{equation}
Thus,
%The last integral can be written in terms of the error function after a complex change of variables so that 
%\[ \int_0^{|\xi|}{e^{i\pi^2y^2}\,dy} = \frac{1}{2\sqrt{\pi}\sqrt{-i}}\,\operatorname{erf}\left( \sqrt{-i}\pi |\xi| \right), \]
%where by the property $\operatorname{erf}(-z) = -\operatorname{erf}(z)$ for $z \in \mathbb{C}$ we can choose any of the two values $\sqrt{-i} = \pm (1-i)/\sqrt{2}$. By convenience we choose $\sqrt{-i} = (1-i)/\sqrt{2}$ so
\begin{equation}
\widehat{g}(\xi) = -\sqrt{2\pi}\,(1+i)\,e^{i\pi^2\xi^2}  + 2\pi^2 |\xi|\left( 1 -\operatorname{erf}\left( \frac{1-i}{\sqrt{2}}\pi |\xi| \right) \right). 
\end{equation}    
The asymptotic expansion of $\operatorname{erfc}(x)$ for $x \in \mathbb R$ is well-known and is obtained integrating its definition by parts $N$ times. 
\end{proof}

Since the error function is analytic, so is $\widehat{g}$. Also, $\operatorname{erfc}(x) = \pi^{-1/2}e^{-x^2}\left( x^{-1} + O(x^{-3}) \right)$, so we get
\begin{equation}\label{DecayForFourierTransformG}
\begin{split}
\widehat{g}(\xi) & = 2\pi^2\,|\xi|\, \pi^{-\frac12}e^{i \pi^2\xi^2}\left( \frac{1+i}{\sqrt{2}} \pi^{-1} |\xi|^{-1} + O(|\xi|^{-3}) \right) - \sqrt{2\pi}\,(1+i)\,e^{i\pi^2 \xi^2} = O(|\xi|^{-2}) \\
%& = \sqrt{2\pi}(1+i)e^{i\pi^2\xi^2} + O(|\xi|^{-2}) - \sqrt{2\pi}(1+i)e^{i\pi^2\xi^2} 
%& = O(|\xi|^{–2})
\end{split}
\end{equation} 
when $|\xi| > 1$. Thus, the hypotheses for the Poisson summation formula are satisfied and \eqref{PoissonSummationFormulaUsed} holds.

Given that $\widehat{g}(0) = \int_{\mathbb{R}}{g(x)\,dx} = -\sqrt{2\pi}\,(1+i)$ and that $g$ and $\widehat{g}$ are even,  Lemma~\ref{thm:Fourier_Transform_Lemma} implies
\begin{equation}
 \phi(h) =  \frac{1+i}{\sqrt{2\pi}}\,\sqrt{h} - \frac{\sqrt{h}}{\pi}\sum_{k=1}^{\infty}{ \left( \frac{\pi k}{ \sqrt{h}}\, \operatorname{erfc}\left( \frac{1-i}{2\sqrt{2}}\, \frac{ k}{ \sqrt{h}} \right) - \sqrt{2\pi}\,(1+i)\,e^{\frac{i k^2}{4h} }  \right) }.
\end{equation}
For each $k \in \mathbb{N}$ and for any  $N \in \mathbb{N}$, the asymptotic expansion of $\operatorname{erfc}$ in \eqref{eq:Error_Function_Asymptotic_Expansion} gives
\begin{equation}
  \frac{\pi k}{ \sqrt{h}}\, \operatorname{erfc}\left( \frac{1-i}{2\sqrt{2}}\, \frac{ k}{ \sqrt{h}} \right) - \sqrt{2\pi} \,  e^{\frac{ik^2}{4h}}\, (1+i)  = \sqrt{\pi}\frac{1+i}{\sqrt{2}} e^{\frac{ik^2}{4h}} \sum_{n=1}^N{ \frac{(2n-1)!!\,2^{n+1}\, h^n}{i^n\,k^{2n}} } + O\left( \frac{\sqrt{h}}{k} \right)^{2N+2} .
\end{equation} 
Sum in $k \in \mathbb{N}$ and change the order of summation to get
\begin{equation}\label{Asymptotic0Preliminary}
\phi(h) =  \frac{1+i}{\sqrt{2\pi}}\,\sqrt{h} - \frac{1-i}{\sqrt{2\pi}}\,\sum_{n=1}^{N}{ \frac{2^{n+1} \, (2n-1)!!}{i^{n-1}}\left( \sum_{k=1}^{\infty}{ \frac{e^{ ik^2/(4h)}}{k^{2n}} } \right) h^{n+\frac12}}  + O\left( h^{N+\frac32} \right)
\end{equation} 
 for any $N \in \mathbb{N}$, which is the asymptotic behavior of $\phi$ around 0.
 
For negative values $h<0$, the property $\phi(-h) = \overline{\phi(h)}$ implies that \eqref{Asymptotic0Preliminary} is correct up to determining $\sqrt{h} = \pm i \sqrt{|h|}$. Indeed, writing $h = -|h| < 0$ and conjugating \eqref{Asymptotic0Preliminary} we have
\begin{equation}%\label{Asymptotic0PreNegative}
\phi(-|h|)  = \frac{1-i}{\sqrt{2\pi}}\,\sqrt{|h|} - \frac{1+i}{\sqrt{2\pi}}\,\sum_{n=1}^{N}{ 2^{n+1}\, i^{n-1} \, (2n-1)!!\left( \sum_{k=1}^{\infty}{ \frac{e^{\frac{-ik^2}{4|h|}}}{k^{2n}} } \right) |h|^{n+\frac12}}  + O\left( h^{N+\frac32} \right),
\end{equation} 
while direct substitution in \eqref{Asymptotic0Preliminary} leads to
\begin{equation}
\phi(-|h|) =   \frac{1+i}{\sqrt{2\pi}}\,\sqrt{-|h|} - \frac{1-i}{\sqrt{2\pi}}\,\sum_{n=1}^{N}{ \frac{2^{n+1} \, (2n-1)!!}{i^{n-1}}\left( \sum_{k=1}^{\infty}{ \frac{e^{\frac{ik^2}{-4|h|}}}{k^{2n}} } \right) (-|h|)^{n+\frac12}}  + O\left( h^{N+\frac32} \right).
\end{equation}
These two expressions coincide if $\sqrt{-1} = -i$, so \eqref{Asymptotic0Preliminary} works also for $h<0$ with the branch of the complex square root with $\sqrt{-1} = -i$.
 
 In short, we have proved the following proposition.
\begin{prop}\label{thm:Asymptotic_At_0}
Let 
\begin{equation}\label{eq:Spiral_0}
Y_n(h) = \sum_{k=1}^{\infty}{ \frac{e^{ik^2/(4h)}}{k^{2n}} }, \qquad n \in \mathbb{N},
\end{equation}
and $N \in \mathbb{N}$. Then, 
\begin{equation}\label{eq:Asymptotic_0_Complete}
\phi(h) =  \frac{1+i}{\sqrt{2\pi}}\,\sqrt{h} - \frac{1-i}{\sqrt{2\pi}}\,\sum_{n=1}^{N}{ \frac{2^{n+1} \, (2n-1)!!}{i^{n-1}}\, Y_n(h) \, h^{n+\frac12}}  + O\left( h^{N+\frac32} \right) 
\end{equation} 
for every $h \in \mathbb{R}$, where $\sqrt{-1} = -i$ if $h<0$. In particular, when $N=1$, we get the self-similar asymptotic expression
\begin{equation}\label{eq:Asymptotic_0_Selfsimilar}
 \phi(h)  = \frac32\,\frac{1+i}{\sqrt{2\pi}}\,\sqrt{h} - 4\pi^2\, \frac{1-i}{\sqrt{2\pi}} \left[ \frac16 - 2\phi\left( \frac{-1}{16\pi^2h} \right) \right] h^{3/2} + O\left( h^{5/2}\right).
\end{equation}
\end{prop}
 
The only thing left to prove is the self-similar expression \eqref{eq:Asymptotic_0_Selfsimilar}, which holds because 
\begin{equation}\label{SelfSimilarity0}
Y_1(h) = \frac{\pi^2}{6} - \frac{i}{8h} - 2\pi^2\,\phi\left( \frac{-1}{16\pi^2h} \right).
\end{equation}
In turn, this last identity is easy to prove using \eqref{FromPhiDuistermaatToPhi}, given that $Y_1(h) = i\pi \phi_D(1/(4\pi h))$.

% Observe that $Y_n(1/\cdot)$ are $8\pi$-periodic, and that both $Y_n$ and $Y_n(1/\cdot)$ follow a circular pattern, since
%\[ Y_n(h) = e^{i/(4h)} + O\left( \sum_{k=2}^{\infty}{ \frac{1}{k^{2n}} } \right) \in B\left( e^{i/(4h)}, \frac{\pi^2}{6}-1 \right) \subset A\left(2-\frac{\pi^2}{6},\frac{\pi^2}{6}\right), \]
%where $B(x,r)$ denotes the ball with center $x$ and radius $r$ and $A(r_1,r_2)$ is the annulus centred in the origin and of radii $r_1 < r_2$.
%Moreover, its limit does not exist when $h \to 0$.

\begin{figure}[h]
  \centering
  \begin{subfigure}[b]{0.45\textwidth}
    \includegraphics[width=\textwidth]{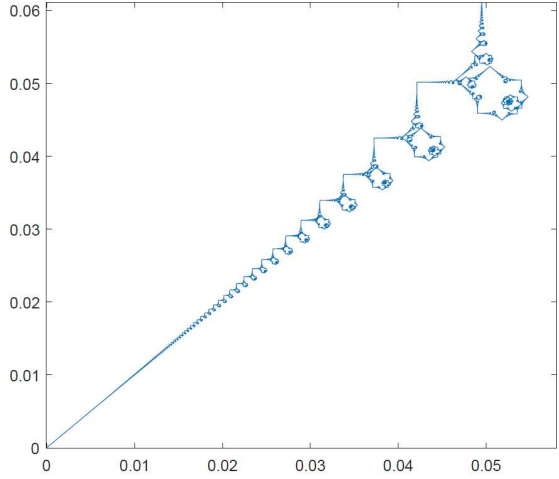}
    \caption{Zoom of $\phi(\mathbb R)$ around $\phi(0)=0$, located on the lower left corner.}
    \label{fig:Zoom0}
  \end{subfigure}
  \hspace{0.05\textwidth}
  \begin{subfigure}[b]{0.45\textwidth}
    \includegraphics[width=\textwidth]{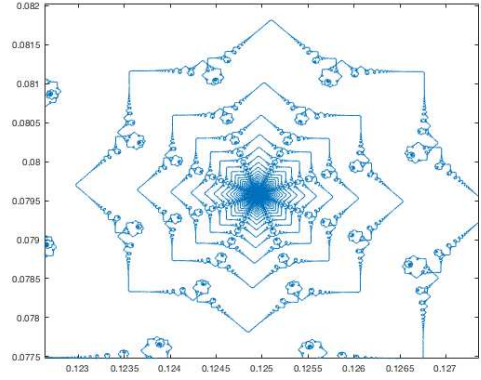}
    \caption{Zoom of $\phi(\mathbb R)$ around $\phi(t_{1,2})$, located in the center of the spiral.}
    \label{fig:Zoom12}
  \end{subfigure}
  \caption{Graphic visualization of the asymptotic behavior of $\phi$ around 0 and $t_{1,2}$. Compare Figure~\ref{fig:Zoom0} to Figure~\ref{FIG_Curva} to appreciate the self-similar patterns, which are analytically explained by \eqref{eq:Asymptotic_0_Selfsimilar} in Proposition~\ref{thm:Asymptotic_At_0}. In Figure~\ref{fig:Zoom12}, the spiraling pattern is a consequence of \eqref{eq:Asymptotic_At_12_Short} in Proposition~\ref{thm:Asymptotic_At_1_2} and the definition of $Z_1$ \eqref{Spiral12}. }
  \label{fig:ZoomBasic} 
\end{figure}

\subsection{Asymptotic behavior around $t_{1,2}$}\label{Section_AsymptoticIn12}

%To compute the asymptotic behaviour of $\phi$ around $t_{1,2}$, one may follow the procedure of Section~\ref{Section_AsymptoticIn0}, first writing
%\[ \phi(t_{1,2}+h) - \phi(t_{1,2}) = - h \, \sum_{k \in \mathbb{Z}} (-1)^k\, g(2\pi k \sqrt{h}) . \]
%The situation is a bit more technical now, but the Poisson summation formula can still be used combined with distribution theory. Indeed, taking into account that $\mathcal{F}(e^{-i\pi (\cdot)}) = \delta_{-1/2} $, we get
%\begin{equation}\label{PoissonIn12}
%\phi(t_{1,2}+h) - \phi(t_{1,2}) = -\frac{\sqrt{h}}{2\pi}\,\sum_{k \in \mathbb{Z}}{\widehat{g}\left( \frac{2k-1}{4\pi \,  \sqrt{h} } \right)}. 
%\end{equation} 
%In this sum, $\widehat{g}$ is never evaluated in 0, so the bound \eqref{DecayForFourierTransformG} can be used to get
%\begin{equation}\label{AsymptoticBound12}
%\left| \phi(t_{1,2}+h) - \phi(t_{1,2}) \right| \leq C_1\,  \sqrt{h}  \sum_{k \in \mathbb{Z}}{ \frac{h}{(2k-1)^2} } = C_2\, h^{3/2},  
%\end{equation} 
%for $h < 1$, where $C_1,C_2 >0$. This is enough to conclude that $\phi$ is differentiable in $t_{1,2}$ with $\phi'(t_{1,2}) = 0$, also that $\phi \in C^{3/2}(t_{1,2})$. This result is equivalent to $\phi_D'(1) = -1/2$, proved by Duistermaat \cite{Duistermaat1991}. The different values are, of course, a consequence of \eqref{FromPhiDuistermaatToPhi}. 
%However, there is a simpler way to proceed rather than using Lemma~\ref{LemmaFourierTransform}. 

An easy way to deduce the asymptotic behavior of $\phi$ around $t_{1,2}$ is by means of the identity 
\begin{equation}\label{IdentityFrom12To0}
\phi(h + t_{1,2}) = \frac18 + \frac{i}{4\pi} + \frac{\phi(4h)}{2} - \phi(h),
\end{equation} 
which can be proved by splitting the sum in the definition of $\phi$ into the even and odd indices. What is more, evaluating it at $h=0$ gives  $\phi(t_{1,2}) = 1/8 + i/(4\pi)$, so
\begin{equation}
\phi( t_{1,2} + h ) - \phi(t_{1,2}) =  \frac{\phi(4h)}{2} - \phi(h).
\end{equation}
We can now use Proposition~\ref{thm:Asymptotic_At_0}. The leading square root terms cancel, so  $h^{3/2}$ becomes the leading order. Moreover, the coefficients of the higher order terms are
\begin{equation}\label{Spiral12}
4^n\,Y_n(4h) - Y_n(h) = 4^n \, Z_n(h), 
%\sum_{\substack{k=1 \\ k \text{ odd}}}^{\infty}{ \frac{e^{ik^2/(16h)}}{k^{2n}} }
 \qquad \text{where} \quad  Z_n(h) = \sum_{\substack{k=1 \\ k \text{ odd}}}^{\infty}{ \frac{e^{ik^2/(16h)}}{k^{2n}} }.
\end{equation}
%These functions are analogous to $Y_n$ and satisfy similar properties, such as periodicity and the circular pattern.
As a consequence, the asymptotic behavior of $\phi$ around $t_{1,2}$ can be written as follows.
\begin{prop}\label{thm:Asymptotic_At_1_2}
Let $N \in \mathbb{N}$. Then, 
\begin{equation}
\phi(t_{1,2}+h) - \phi(t_{1,2}) = -\frac{1-i}{\sqrt{2\pi}}\, \sum_{n=1}^N{ \frac{2^{3n+1}\,(2n-1)!!}{i^{n-1}}\, Z_n(h)\, h^{n+\frac12} } + O\left( h^{N+\frac32} \right)
\end{equation}   
for every $h \in \mathbb{R}$, where $\sqrt{-1} = -i$ when $h<0$. In particular, when $N=1$, 
\begin{equation}\label{eq:Asymptotic_At_12_Short}
 \phi(t_{1,2}+h) - \phi(t_{1,2}) = -16\,\frac{1-i}{\sqrt{2\pi}}\,Z_1(h)\,h^{3/2} + O\left( h^{5/2} \right). 
\end{equation}
\end{prop}

\begin{rem}
The function $Z_1(h)$ turns around the origin in a circular pattern, and the more $h$ approaches to zero, the faster it does it. Since in \eqref{eq:Asymptotic_At_12_Short} it is multiplied by $h^{3/2}$, which tends to zero when $h \to 0$, this circular pattern turns into a spiral that concentrates in $\phi(t_{1,2})$ (see Figure~\ref{fig:Zoom12}). 
\end{rem}

\begin{rem}
Identities similar to \eqref{IdentityFrom12To0} can be obtained for other rationals such as $t_{1,3}, t_{1,4},t_{1,6}$ and $t_{1,8}$. Consequently, one can prove the asymptotic behavior of $\phi$ around those points with as much precision as wanted. 
\end{rem}

\section{Asymptotic behavior around rationals}\label{Section_Rationals}

Once we know the asymptotic behavior around 0 and $t_{1,2}$, we compute the case of a general rational $t_{p,q}$. For that, we make the reduction process explained in Subsection~\ref{Section_LookingForTransformations} rigorous. First of all, the formal identity \eqref{AsymptoticWithTheta} in which the reduction is based is made precise by
\begin{equation}
\phi(t) = i\, \lim_{\epsilon \to 0^+} \int_0^t \theta(-4\pi \tau + i\epsilon)\,d\tau.
\end{equation}
This is a consequence of Fubini's theorem and the dominated convergence theorem. Consequently, we get the rigorous version of \eqref{eq:Asymptotic_Ready_For_Reduction_With_Modular_Group_Rationals},
\begin{equation}\label{eq:Asymptotic_With_Limit_2}
 \phi\left( t_{p,q} + h \right) - \phi(t_{p,q}) = \frac{i}{4\pi}\,\lim_{\epsilon \to 0^+} \int_{\tilde{p}/\tilde{q}}^{\tilde{p}/\tilde{q} + 4\pi h}{\theta(-\tau + i\epsilon)\,d\tau}.
\end{equation}
Let now $\gamma \in \Gamma_\theta$ and use the transformation \eqref{eq:Theta_Transformation_General} for the Jacobi $\theta$ function so that, after conjugation, \eqref{eq:Asymptotic_With_Limit_2} turns into
\begin{equation}\label{eq:Asymptotic_At_Q013_Before_Integrating_By_Parts}
\overline{ \phi(t_{p,q} + h) - \phi(t_{p,q}) } = \frac{1}{4\pi i e_\gamma}\lim_{\epsilon\to 0^+}\int_{\tilde{p}/\tilde{q}}^{ \tilde{p}/\tilde{q} + 4\pi h } \frac{ \theta(\gamma(\tau + i \epsilon)) }{\sqrt{ c(\tau + i\epsilon) + d }  }\,d\tau.
\end{equation}
Observing that $\phi'(z) = i\theta(-4\pi z)$ whenever $\operatorname{Im}z >0$, integrate by parts choosing
\begin{equation}
\begin{array}{ll}
u = \frac{1}{\gamma'(\tau + i \epsilon)\, \sqrt{c(\tau + i \epsilon) + d}} = (c(\tau + i \epsilon) + d)^{3/2}, &   du = \frac{3c}{2}\, \sqrt{c(\tau + i \epsilon) + d}\,d\tau,  \\
dv = \theta(\gamma(\tau + i\epsilon)))\,\gamma'(\tau + i\epsilon)\,d\tau, & v = 4\pi i\, \phi\left(-\frac{\gamma(\tau + i\epsilon)}{4\pi}\right),
\end{array}
\end{equation}
which yields
\begin{equation}
 \frac{1}{e_\gamma}\,\lim_{\epsilon \to 0}\Bigg[\phi\left(-\frac{\gamma(\tau + i\epsilon)}{4\pi}\right)\,(c(\tau + i \epsilon) + d)^{\frac32} \Bigg|_{\tilde{p}/\tilde{q}}^{\tilde{p}/\tilde{q} + 4\pi h}   - \frac{3c}{2}\, \int_{\tilde{p}/\tilde{q}}^{\tilde{p}/\tilde{q} + 4\pi h} \phi\left(-\frac{\gamma(\tau + i\epsilon)}{4\pi}\right)\,  \sqrt{c(\tau + i \epsilon) + d} \,d\tau \Bigg].
\end{equation}
This allows to work exclusively with $\phi$, which is well-defined on the real line. Clearly, we can now take the limit $\epsilon \to 0$ in the first term. In the second term, due to the fact that the integrating interval is finite, everything inside the integral is bounded independently of $\epsilon$. Thus, the limit can be taken inside by the theorem of dominated convergence to get
\begin{equation}\label{eq:Valid_For_Any_Gamma}
 \phi(t_{p,q} + h)  - \phi(t_{p,q})   = e_\gamma \,\Bigg[\phi\left(\frac{\gamma(\tau)}{4\pi}\right)\,(c\tau  + d)^{3/2} \Bigg|_{\tilde{p}/\tilde{q}}^{\tilde{p}/\tilde{q} + 4\pi h}  - \frac{3c}{2}\, \int_{\tilde{p}/\tilde{q}}^{\tilde{p}/\tilde{q} + 4\pi h} \phi\left(\frac{\gamma(\tau)}{4\pi}\right)\,  \sqrt{c\tau  + d} \,d\tau \Bigg].
\end{equation}

%
%\begin{rem}\label{rem:Valid_For_Any_Gamma}
%We remark here that identity \eqref{eq:Valid_For_Any_Gamma} is valid for every $\gamma \in \Gamma_\theta$, since we still have not used any particular property of the transformation that $\gamma(\tilde{p}/\tilde{q}) = 0$. 
%\end{rem}

\subsection{Asymptotic behavior around $t_{p,q}$ with $q \equiv 0,1,3 \pmod{4}$}\label{Section_AsymptoticInQ013}

	Let $p/q$ be an irreducible fraction such that $q \equiv 0,1,3 \pmod{4}$. In Subsection~\ref{Section_LookingForTransformations} we found $\gamma \in \Gamma_{\theta}$ such that $\gamma(\tilde{p}/\tilde{q}) = 0$, where $\tilde{p}/\tilde{q} = 2p/q$. Recalling \eqref{eq:Upper_Boundary_Of_Integral}, the definition of $b(h)$ in \eqref{eq:Definition_Of_B} and $c = c_\pm$, 
%the first boundary term is
%\begin{equation}
% \phi(-b(h))\,\frac{(1 + 4\pi c_\pm \tilde{q}h)^{3/2}}{\tilde{q}^{3/2}} 
% % = \phi(-b(h))\,\Big(c_\pm \gamma^{-1}(4\pi b(h))+d\Big)^{3/2} 
% = \frac{\phi(-b( h))}{ (\tilde{q} - 4\pi c_\pm b(h))^{3/2} }
%\end{equation}
%The second boundary term corresponds to evaluating the first one in $h=0$. Of course, it is null because $\gamma(\tilde{p}/\tilde{q})=0$. Hence, 
\begin{equation}
\phi(t_{p,q} + h) - \phi(t_{p,q})  = e_\gamma\,  \Bigg[  \frac{(1+4\pi c_\pm \tilde{q} h)^{3/2}}{\tilde{q}^ {3/2}}  \, \phi\left( b(h) \right)   - \frac32 c_\pm \, \int_{\tilde{p}/\tilde{q}}^{\tilde{p}/\tilde{q} + 4\pi h}{ \phi\left( \frac{\gamma(\tau)}{4\pi} \right) \sqrt{c_\pm \tau + d}  \,d\tau} \Bigg]
\end{equation}
Change variables $\gamma(\tau)/4\pi = r$ as in \eqref{eq:Changing_Variables} to get
\begin{equation}\label{eq:Asymptotic_Closed_Form_013}
\phi(t_{p,q} + h) - \phi(t_{p,q}) =  e_\gamma \left[ \frac{ \phi\left( b(h) \right)}{\left(\tilde{q} - 4\pi c_\pm b(h)\right)^{3/2}} - 6\pi c_{\pm}\,\int_0^{b(h)}{ \frac{\phi(r)}{(\tilde{q}-4\pi c_{\pm} r)^{5/2}}\,dr } \right].
\end{equation}

We can already use the asymptotic behavior around 0 in $\phi(b(h))$ and $\phi(r)$ because $b(h)$ behaves like $\tilde{q}^2h$ when $h$ is small. For simplicity, call $b=b(h)$. Develop $(\tilde{q} - 4\pi c_\pm b)^{-3/2}$ and $(\tilde{q} - 4\pi c_\pm b)^{-5/2}$ using the Taylor series 
\begin{equation}\label{TaylorSeries}
(1-x)^{-\alpha} = \sum_{n=0}^{\infty}{ \binom{n + \alpha - 1}{n}x^n },  \qquad |x| < 1,
\end{equation} 
 which can be done because $4\pi c_{\pm}b(h)/\tilde{q} < 1$ for all $h \in \mathbb{R}$. Also, develop $\phi(b)$ following Proposition~\ref{thm:Asymptotic_At_0} so that we get   
\begin{equation}\label{AsymptoticsWithB013}
\phi(t_{p,q} + h) - \phi(t_{p,q}) = \frac{e_{\gamma}}{\tilde{q}^{3/2}}  \left[ \frac{1+i}{\sqrt{2\pi}}\,b^{\frac12} + \left( 2\pi\,\frac{1+i}{\sqrt{2\pi}}\,\frac{c_{\pm}}{\tilde{q}} - 4\,\frac{1-i}{\sqrt{2\pi}}\,Y_1(b) \right)\,b^{\frac32} + O\left( b^{\frac52}\right) \right]. % \quad \forall t \in \mathbb{R},
\end{equation} 
Computing further terms requires integrating $r^{3/2}Y_1(r)$. Using  \eqref{TaylorSeries} again, expand 
\begin{equation}\label{TaylorSeriesOfB}
\begin{split}
& b(h)^{1/2} = \tilde{q}\,h^{1/2} \left( 1 - 2\pi\,c_{\pm}\tilde{q}h + O\left(c_{\pm}^2\,\tilde{q}^2\,|h|^{3/2} \right) \right), \qquad  b(h)^{3/2} = \tilde{q}^3\,h^{3/2} \left(1 + O\left(\tilde{q}\,|c_{\pm}h | \right)\right), \\
& \qquad \qquad \qquad \qquad \qquad \qquad \qquad  b^{5/2}(h) =O\left(q^5|h|^{5/2} \right),
\end{split}
\end{equation} 
which according to the definition of $b(h)$ are valid only if $4\pi |c_{\pm}\tilde{q} h| < 1$. We use them to expand \eqref{AsymptoticsWithB013} in terms of $h$ and obtain
\begin{equation}
 \phi(t_{p,q}+h) - \phi(t_{p,q}) = e_{\gamma}\,\left(  \frac{1+i}{\sqrt{2\pi}}\,\frac{h^{1/2}}{\tilde{q}^{1/2}} -4\,\frac{1-i}{\sqrt{2\pi}}\,Y_1(b(h))\,\tilde{q}^{3/2}\,t^{3/2} + O\left( \tilde{q}^{\frac72}\,h^{\frac52} \right) \right), %\quad  |h| < \frac{1}{4\pi \frac{c_{\pm}}{\tilde{q}}}\,\frac{1}{\tilde{q}^2}. 
\end{equation} 
valid for $\tilde{q}^2\,h < 1/(4\pi c_+/\tilde{q})$ when $h >0$ and for $\tilde{q}^2\,|h| < 1/(4\pi |c_-|/\tilde{q})$ when $h <0$. This is the asymptotic behavior we looked for, which we write in the following proposition:

\begin{prop}\label{thm:Asymptotic_At_Q013}
Let $p, q \in \mathbb{N}$ such that $q \equiv 0,1,3 \, (\text{mod } 4)$, $p<q$ and $\operatorname{gcd}(p,q) = 1$. Define $\tilde{p}$ and $\tilde{q}$ so that $\tilde{p}/\tilde{q} = 2p/q$ is an irreducible fraction, and set 
\begin{equation}
Y_1(h) = \sum_{k=1}^{\infty}{ \frac{e^{ik^2/(4h)}}{k^{2}} } \qquad \text{ and } \qquad b(h) = \left\{
\begin{array}{ll}
\frac{\tilde{q}^2h}{1+4\pi c_{+} \tilde{q} h}, & \text{when } h \geq 0, \\
\frac{\tilde{q}^2h}{1+4\pi c_{-} \tilde{q} h}, & \text{when } h < 0,
\end{array}
\right.
\end{equation}  
where $\tilde{q} \leq c_+, |c_-| \leq 4\tilde{q}$ as in Subsection~\ref{Section_LookingForTransformations}. Then, there exists a complex eighth root of unity $e_{p,q}$ depending only on $p$ and $q$ such that
\begin{equation}\label{eq:Asymptotic_At_Q013_Principal}
\phi(t_{p,q}+h) -  \phi(t_{p,q})  = \frac{e_{p,q}}{\sqrt{\pi}}\frac{1+i}{\sqrt{2}}\,\Bigg( \frac{h^{1/2}}{\tilde{q}^{1/2}}  + 4\,i\,Y_1(b(h))\,\tilde{q}^{3/2}\,h^{3/2}  + O\left( \tilde{q}^{7/2}\,h^{5/2} \right) \Bigg),
\end{equation}    
which is valid when $|h| \leq 1/(4\pi\frac{|c_\pm|}{\tilde{q}}\tilde{q}^2)$ and where $c_\pm = c_+$ when $h>0$ and   $c_\pm = c_-$ when $h<0$. Also, $\sqrt{-1} = -i$ when $h < 0$. The corresponding the self-similar form is
\begin{equation}\label{eq:Asymptotic_At_Q013_Selfsimilar}
\begin{split}
& \phi(t_{p,q}+h) - \phi(t_{p,q}) \\
& \qquad \qquad  = \frac32\,\frac{e_{p,q}}{\sqrt{\pi}}\frac{1+i}{\sqrt{2}}\,\Bigg[ \frac{h^{1/2}}{\tilde{q}^{1/2}} + \frac{8\pi^2}{3} i \left( \frac16 - \frac{i}{2\pi}\,\frac{c_\pm}{\tilde{q}} - 2\phi\left( \frac{-1}{16\pi^2b(h)} \right) \right)\,\tilde{q}^{\frac32}\,h^{\frac32} + O\left( \tilde{q}^{\frac72}\,h^{\frac52} \right) \Bigg] 
\end{split}
\end{equation} 
for the same values $|h| \leq 1/(4\pi\frac{|c_\pm|}{\tilde{q}}\tilde{q}^2)$ as above. 
Also equivalently, the above is rescaled as
\begin{equation}\label{eq:Asymptotic_At_Q013_Rescaled}
\begin{split}
&  \phi\left(t_{p,q}+\frac{h}{\tilde{q}^2}\right) -  \phi(t_{p,q}) \\
& \qquad \qquad  = \frac{3}{2\sqrt{\pi}}\,\frac{1+i}{\sqrt{2}}\,\frac{e_{p,q}}{\tilde{q}^{3/2}}  \, \left[ h^{\frac12} + \frac{8\pi^2}{3} i \left( \frac16 - \frac{i}{2\pi}\,\frac{c_\pm}{\tilde{q}} - 2\phi\left( \frac{-1}{16\pi^2\beta(h)} \right) \right)h^{\frac32}  + O\left( h^{\frac52} \right) \right] 
\end{split}
\end{equation}
for all $|h| \leq 1/(4\pi\frac{|c_\pm|}{\tilde{q}})$, where $\beta(h) = b(h/\tilde{q}^2)$. 
\end{prop}

\begin{rem}
 The leading square root term is the cause of every right-angled corner in Figure~\ref{FIG_Curva}, since $\sqrt{-1} = \pm i$. Also, the self-similar patterns of $\phi$ in Figure~\ref{FIG_Curva} are analytically explained by the term $\phi(-1/(16\pi^2b(h)))$ in the expansions \eqref{eq:Asymptotic_At_Q013_Selfsimilar} and \eqref{eq:Asymptotic_At_Q013_Rescaled}. In fact, \eqref{eq:Asymptotic_At_Q013_Selfsimilar} is obtained from \eqref{eq:Asymptotic_At_Q013_Principal} via the identity \eqref{SelfSimilarity0} that we already used in the previous section. 
\end{rem}

\begin{rem}
Comparing  \eqref{eq:Asymptotic_At_Q013_Rescaled} with \eqref{eq:Asymptotic_0_Selfsimilar} in Proposition~\ref{thm:Asymptotic_At_0}, we see that $\phi$ behaves around $t_{p,q}$ essentially the same way as around 0, except rescaling the variable by $\tilde{q}^{-2}$ and the image by $\tilde{q}^{3/2}$ and replacing $h$ with $\beta(h)$ in the self-similar term. This is the rigorous version of  \eqref{eq:Approximated_Asymptotic_013} that we anticipated formally in Subsection~\ref{Section_LookingForTransformations}.
\end{rem}

\begin{rem}\label{ChoiceOfSquareRootQ013}
In Proposition~\ref{thm:Asymptotic_At_Q013}, we claim $\sqrt{-1} = -i$ whenever  $h<0$. The symmetry $\phi(-t) = \overline{\phi(t)}$ was enough to determine this around 0, but there is no such symmetry around $\phi(t_{p,q})$ for $q > 2$. However, we can work with the limit $h \to 0$ in the asymptotic expression of $\phi(t_{p,q}+h) - \phi(t_{p,q})$. 

Let $0 < |h| \ll 1$. We start with \eqref{eq:Asymptotic_Closed_Form_013}, where the leading term when $h \to 0$ is the first one. Indeed,  $\lim_{h\to 0}b(h) = 0$, so by Proposition~\ref{thm:Asymptotic_At_0} we have
\begin{equation}
\frac{\phi(b(h))}{(\tilde{q} - 4\pi c_\pm b(h))^{3/2}} \sim \frac{\textstyle{\frac{1+i}{\sqrt{2\pi}}}\,b^{1/2}+ O(b^{3/2})}{\tilde{q}^{3/2}} \qquad \text{ when } h \to 0, 
\end{equation}  
and
\begin{equation}
\int_0^b\frac{\phi(r)}{(\tilde{q}-4\pi c_\pm r)^{5/2}}\,dr \sim \frac{1+i}{\sqrt{2\pi}} \int_0^b \frac{  r^{1/2}+ O(r^{3/2})}{\tilde{q}^{5/2}}\,dr = \frac{1+i}{\sqrt{2\pi}} \frac{ b^{3/2} + O(b^{5/2})}{\tilde{q}^{5/2}} \qquad \text{ when } h \to 0.
\end{equation}   
Consequently, %from \eqref{eq:Asymptotic_Closed_Form_013} we get
\begin{equation}\label{eq:Aux_Short_Asymptotic}
1 = \lim_{h \to 0}\frac{ \phi(t_{p,q} + h) - \phi(t_{p,q}) }{e_{p,q}\,\tilde{q}^{-3/2}\,\phi(b(h))}. 
\end{equation}  
Define $b_-(h)$ by
\begin{equation}
b(- h ) = -\frac{\tilde{q}^2 h}{1 + 4\pi c_- \tilde{q}h} = - b_-(h), 
%\quad \text{ where } b_-(|h|) = \frac{\tilde{q}^2|h|}{1 + 4\pi c_- \tilde{q}|h|},
\end{equation} 
so that $ \overline{\phi(b(-h))} = \phi(b_-(h) $. Therefore, evaluate \eqref{eq:Aux_Short_Asymptotic} in $-h$ and conjugate it so that
\begin{equation}
\begin{split}
1 & = e_{p,q} \,\lim_{h \to 0} \frac{ \overline{ \phi(t_{p,q} - h) - \phi(t_{p,q}) }}{\tilde{q}^{- 3/2} \, \overline{\phi(b(-h))} } = e_{p,q} \, \lim_{h \to 0} \frac{\overline{\phi(t_{p,q} - h) - \phi(t_{p,q})}}{\tilde{q}^{-3/2}\, \phi(b_-(h)) } \\
&  = e_{p,q} \, \,\lim_{h \to 0} \frac{\overline{\phi(t_{p,q} - h) - \phi(t_{p,q})}}{\tilde{q}^{-3/2}\, \phi(b(h)) }\,  \frac{\phi(b(h))}{ \phi(b_-(h)) } =  e_{p,q} \, \,\lim_{h \to 0} \frac{\overline{\phi(t_{p,q} - h) - \phi(t_{p,q})}}{\tilde{q}^{-3/2}\, \phi(b(h)) } \\
& = e_{p,q}^2 \,\lim_{h \to 0} \frac{\overline{\phi(t_{p,q} - h) - \phi(t_{p,q})}}{\phi(t_{p,q} + h) - \phi(t_{p,q})}.
\end{split}
\end{equation}
We used \eqref{eq:Aux_Short_Asymptotic} in the last equality, and  
\begin{equation}
\lim_{h \to 0} \frac{\phi(b(h))}{\phi(b_-(h))} = \lim_{h \to 0} \frac{b(h)^{1/2}}{b_-(h)^{1/2}} = 1.   
\end{equation} 
in the previous one. Finally, using the asymptotic behavior in Proposition~\ref{thm:Asymptotic_At_Q013}, we get
\begin{equation}
1 = e_{p,q}^2 \lim_{h \to 0} \frac{ \overline{ e_{p,q} (1+i) (-h)^{1/2} }  }{ e_{p,q} (1+i) h^{1/2}  } = e_{p,q}^2 \frac{\overline{e_{p,q}}}{e_{p,q}} \frac{1-i}{1+i} \overline{\sqrt{-1}} = -i \overline{\sqrt{-1}},
\end{equation}
which implies that $\sqrt{-1} = -i$ must hold so that Proposition~\ref{thm:Asymptotic_At_Q013} works also for $h<0$. 
\end{rem}

As a corollary, we show that the asymptotic behavior in Proposition~\ref{thm:Asymptotic_At_Q013} can be truncated in its first term independently of $q$, which is what we use in the proofs of Theorems~\ref{TheoremHausdorffDimension} and \ref{TheoremHausdorffDimensionGeneralised}.
\begin{cor}\label{thm:Global_Bound_Q013}
Let $p, q \in \mathbb{N}$ such that $q \equiv 0,1,3 \pmod{4}$ and $\operatorname{gcd}(p,q) = 1$. Given $M >0$, there exists $C_M>0$ independent of $p$ and $q$ such that 
\begin{equation}
\left|  \phi(t_{p,q}+h) - \phi(t_{p,q}) \right|  \leq C_M \frac{|h|^{1/2}}{q^{1/2}}, \qquad \forall |h| < \frac{M}{q^2}.
\end{equation}
\end{cor}
\begin{proof}
The Taylor expansion that was used to get \eqref{AsymptoticsWithB013} works because $4\pi c_{\pm}b/\tilde{q} < 1$ for all $h \in \mathbb{R}$. However,  $\lim_{h\to\infty}4\pi c_\pm b(h)/\tilde{q} = 1$, so we can truncate the series uniformly only if $4\pi c_\pm b(h)/\tilde{q} < \delta$ for some fixed $0<\delta<1$. That is equivalent to $|h| < (\textstyle{\frac{\delta}{4\pi \frac{|c|}{\tilde{q}}(1-\delta)}})/\tilde{q}^2$. 

Now, given $M>0$, since $\delta/(1-\delta)$ covers the whole positive real line for $\delta \in (0,1)$, there exists $0<\delta_M < 1$ such that $M = \textstyle{\frac{\delta_M}{16\pi(1-\delta_M)}}$. Since $|c_\pm |< 4\tilde{q}$, then $|h| < M/\tilde{q}^2$ means that $4\pi c_\pm b(h)/\tilde{q} < \delta_M$, and thus we can truncate \eqref{AsymptoticsWithB013}, in the sense that there exists $C_{\delta_M}>0$ such that 
\begin{equation}\label{eq:Truncation_In_Terms_Of_B}
 \left|  \phi(t_{p,q}+h) - \phi(t_{p,q}) \right|  \leq C_{\delta_M} \frac{|b(h)|^{1/2}}{\tilde{q}^{3/2}}.
\end{equation}
Now, if $4\pi |c_{\pm} \tilde{q} h| \geq 1$, then from the definition of $b(h)$ we have $|b(h)| \leq \tilde{q}^2 |h| /2$, so we get 
\begin{equation}
 \left|  \phi(t_{p,q}+h) - \phi(t_{p,q}) \right|  \leq \frac{C_{\delta_M}}{\sqrt{2}} \frac{|h|^{1/2}}{\tilde{q}^{1/2}}.
\end{equation}
Otherwise, if $4\pi |c_{\pm} \tilde{q} h| < 1$, then the bound is immediate from Proposition~\ref{thm:Asymptotic_At_Q013} because in particular we have $|h| < \tilde{q}^{-2}$ and then 
\begin{equation}
\tilde{q}^{7/2}|h|^{5/2} < \tilde{q}^{3/2}|h|^{3/2}  < \tilde{q}^{-1/2}|h|^{1/2}
\end{equation} 
can be used in \eqref{eq:Asymptotic_At_Q013_Principal}. 
\end{proof}

\begin{figure}[h]
  \centering
    \includegraphics[width=0.7\textwidth]{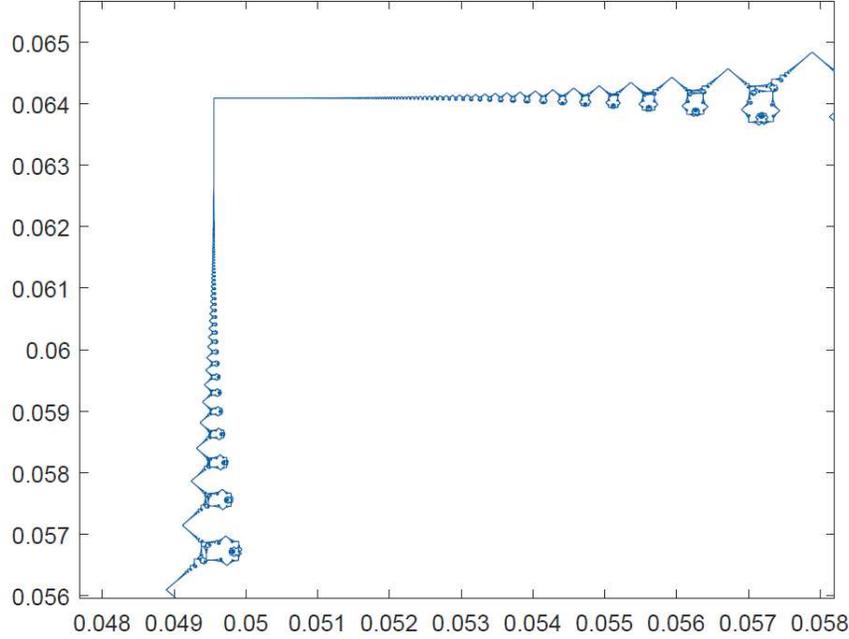}
    \caption{Zoom of $\phi(\mathbb R)$ around $\phi(t_{1,8})$. Compare to Figure~\ref{FIG_Curva} to appreciate the self-similar pattern, which is analytically explained in \eqref{eq:Asymptotic_At_Q013_Selfsimilar} in Proposition~\ref{thm:Asymptotic_At_Q013}. Compare it also to the behavior of $\phi$ around 0 in Figure~\ref{fig:Zoom0}. Except for a rotation by $\pi/4$ radians, they are very similar. }
  \label{fig:Zoom18} 
\end{figure}

\subsection{Asymptotic behavior around $t_{p,q}$ with $q \equiv 2 \pmod{4}$}\label{Section_AsymptoticInQ2}
If $p/q$ is an irreducible fraction such that $q \equiv 2 \pmod{4}$, we saw that there exists $\gamma \in \Gamma_{\theta}$ satisfying $\gamma(\tilde{p}/\tilde{q}) = 1$, where $\tilde{p}/\tilde{q} = 2p/q$ is irreducible. The strategy is exactly the same as in Subsection~\ref{Section_AsymptoticInQ013}, except that when integrating by parts in \eqref{eq:Asymptotic_At_Q013_Before_Integrating_By_Parts} we choose  
\begin{equation}
v = 4\pi i \, \left[ \phi\left(-\frac{\gamma(\tau+i\epsilon) }{4\pi} \right) - \phi\left(-\frac{\gamma(\tilde{p}/\tilde{q} + i\epsilon)}{4\pi} \right) \right]
\end{equation}
instead. Then, after taking the limit $\epsilon \to 0$ and changing variables $\gamma(\tau)/4\pi=r$ as before, we get
\begin{equation}\label{eq:Asymptotic_Closed_Form_2}
\phi(t_{p,q} + h) - \phi(t_{p,q})   = e_{\gamma} \Bigg[ \frac{\phi(t_{1,2} + b(h)) - \phi(t_{1,2})}{(\tilde{q}-4\pi c_{\pm} b(h))^{3/2}}   - 6\pi \,c_{\pm} \int_0^{b(h)}{  \frac{\phi(t_{1,2} + r) - \phi(t_{1,2})}{(\tilde{q}-4\pi c_{\pm} r)^{5/2}}\,dr   }  \Bigg]
\end{equation} 
for all $h \in \mathbb{R}$. Now develop $\phi(t_{1,2} + b(h)) - \phi(t_{1,2})$ using Proposition~\ref{thm:Asymptotic_At_1_2} and use the Taylor expansions \eqref{TaylorSeries} to get a series in terms of $b = b(h)$,
\begin{equation}\label{eq:Asymptotic_With_B_2}
\phi(t_{p,q}+h) - \phi(t_{p,q}) = e_{\gamma} \,\left[  -16\,\frac{1-i}{\sqrt{2\pi}}\, Z_1(b)\,\frac{b^{3/2}}{\tilde{q}^{3/2}} + \frac{1}{\tilde{q}^{3/2}}\,O\left( b^{5/2} \right)   \right].
\end{equation} 
Finally, expanding the Taylor series for powers of $b(h)$ as in \eqref{TaylorSeriesOfB}, we get the asymptotic behavior we were looking for:

\begin{prop}\label{thm:Asymptotic_At_Q2}
Let $p, q \in \mathbb{N}$ such that $q \equiv 2 \pmod{4}$, $p<q$ and $\operatorname{gcd}(p,q) = 1$. Define $\tilde{p}$ and $\tilde{q}$ so that $\tilde{p}/\tilde{q} = 2p/q$ is an irreducible fraction, and set
\begin{equation}
Z_1(h) = \sum_{\substack{k=1 \\ k \text{ odd}}}^{\infty}{ \frac{e^{ik^2/(16h)}}{k^{2}} } \qquad \text{ and } \qquad b(h) = \left\{
\begin{array}{ll}
\frac{\tilde{q}^2h}{1+4\pi c_{+} \tilde{q} h}, & \text{when } h \geq 0, \\
\frac{\tilde{q}^2h}{1+4\pi c_{-} \tilde{q} h}, & \text{when } h < 0,
\end{array}
\right.
\end{equation}
where $\tilde{q} \leq c_+, |c_-| \leq 3\tilde{q}$ as in Subsection~\ref{Section_LookingForTransformations}. Then, there exists a complex eighth root of unity $e_{p,q}$  depending only on $p$ and $q$ such that
\begin{equation}\label{eq:Asymptotic_At_Q2_Principal}
\phi(t_{p,q}+h) - \phi(t_{p,q})   = e_{p,q} \,\left(  -16\,\frac{1-i}{\sqrt{2\pi}}\,Z_1(b(h))\, \tilde{q}^{3/2}\,h^{3/2} + O\left( \tilde{q}^{7/2} h^{5/2} \right)   \right), \quad  |h| < \frac{1}{4\pi \frac{|c_\pm|}{\tilde{q}}}\,\frac{1}{\tilde{q}^2},
\end{equation}
where $c_\pm = c_+$ when $h>0$ and $c_\pm = c_-$ when $h<0$. Also, $\sqrt{-1} = -i$ when $h < 0$. Equivalently, rescaling the variable, 
\begin{equation}\label{eq:Asymptotic_At_Q2_Rescaled}
\phi\left(t_{p,q}+\frac{h}{\tilde{q}^2}\right) - \phi(t_{p,q})  = \frac{e_{p,q}}{\tilde{q}^{3/2}} \,\left(  -16\,\frac{1-i}{\sqrt{2\pi}}\,Z_1(\beta(h)) \,h^{3/2} + O\left(  h^{5/2} \right)   \right),  \qquad    |h| < \frac{1}{4\pi \frac{|c_\pm|}{\tilde{q}}},
\end{equation}
where $\beta(h) = b(h/\tilde{q}^2)$.
\end{prop}

%\begin{rem}
%The branch of the square root with $\sqrt{-1} = -i$ is the correct one for the same reasons as in Remark~\ref{ChoiceOfSquareRootQ013}.
%\end{rem}

\begin{rem}
Proposition~\ref{thm:Asymptotic_At_Q2} confirms what we formally deduced in\eqref{eq:Approximated_Asymptotic_2}, this is, that $\phi$ behaves around $t_{p,q}$ with $q \equiv 2 \pmod{4}$ essentially the same way as around $t_{1,2}$, except the usual rescaling and replacing $h$ with $\beta(h)$ in the argument of $Z_1$. 
\end{rem}

%\begin{rem}
%The circular pattern of $Z_1$ makes the leading term form a spiral-like behavior, which can be observed in Figure~\ref{FIG_Curva}.
%\end{rem}

The analogous result of Corollary~\ref{thm:Global_Bound_Q013} is also satisfied, with an equally analogous proof.
\begin{cor}\label{thm:Global_Bound_Q2}
Let $p, q \in \mathbb{N}$ such that $q \equiv 2 \pmod{4}$ and $\operatorname{gcd}(p,q) = 1$. Given $M>0$, there exists $C_M>0$ independent of $p$ and $q$ such that  
\begin{equation}
\left| \phi(t_{p,q}+h) - \phi(t_{p,q}) \right|  \leq C_M \,  q^{3/2}\, h^{3/2}, \qquad \forall |h| \leq \frac{M}{q^2}. 
\end{equation}  
\end{cor}

\bibliographystyle{acm}
\bibliography{Bibliography_2020_11.bib}

\end{document}